\documentclass[11pt]{amsart}
\usepackage{cite}
\usepackage{varwidth}
\usepackage[english, ngerman]{babel}
\usepackage[T1]{fontenc}
\usepackage[final]{microtype}
\usepackage{ragged2e}
\usepackage[ansinew]{inputenc}
\usepackage[normalem]{ulem}
\usepackage{csquotes}
\usepackage{chngcntr}
\usepackage{lmodern}
\usepackage{color}
\usepackage{textcomp}
\usepackage{comment}
\usepackage[arrow, matrix, curve]{xy}
\usepackage{mathtools}
\usepackage{mathrsfs}
\usepackage{leftidx}
\usepackage{enumitem}
\usepackage{epic}
\usepackage{placeins}
\usepackage{mathtools}
\usepackage{amsfonts}
\usepackage{amssymb}
\usepackage{amsthm}
\usepackage{url}

\numberwithin{paragraph}{section}
\numberwithin{equation}{section}

\DeclareMathOperator{\Spf}{Spf}
\DeclareMathOperator{\Spec}{Spec}

\DeclareMathOperator{\val}{val}
\DeclareMathOperator{\Div}{div}

\DeclareMathOperator{\red}{red}

\DeclareMathOperator{\supp}{supp}

\DeclareMathOperator{\Id}{Id}
\DeclareMathOperator{\MA}{MA}

\DeclareMathOperator{\Deg}{deg}
\DeclareMathOperator{\cyc}{cyc}
\DeclareMathOperator{\codim}{codim}
\DeclareMathOperator{\irr}{irr}

\DeclareMathOperator{\Int}{Int}
\DeclareMathOperator{\Quot}{Quot}

\def\quotient#1#2{\raise0.75ex\hbox{$\,#1$}\big/\lower0.75ex\hbox{$#2\,$}}

\title{A comparison of the real and non-archimedean Monge-Ampère operator}

\author[C.~Vilsmeier]{Christian Vilsmeier}
\address{C. Vilsmeier, Mathematik, Universit{\"a}t 
Regensburg, 93053 Regensburg, Germany}
\email{Christian.Vilsmeier@mathematik.uni-regensburg.de}

\begin{document}
	\theoremstyle{definition}
	\newtheorem{Def}{Definition}[section]
	\theoremstyle{definition}
	\newtheorem{Example}[Def]{Example}
	\theoremstyle{remark}
	\newtheorem{Remark}[Def]{Remark}
	\newtheorem{Construction}[Def]{Construction}
	\theoremstyle{plain}
	\newtheorem{Lemma}[Def]{Lemma}
	\newtheorem{Proposition}[Def]{Proposition}
	\newtheorem{Theorem}[Def]{Theorem}
	\newtheorem{Corollary}[Def]{Corollary}
	%\counterwithout{equation}{chapter}
	
	\selectlanguage{english}
	
	\let\thefootnote\relax\footnotetext{This work was partially supported by the collaborative research center 'SFB 1085: Higher Invariants' funded by the Deutsche Forschungsgemeinschaft.}
	
	\begin{abstract}
		Let $X$ be a proper algebraic variety over a non-archimedean, non-trivially valued field. We show that the non-archimedean Monge-Ampère measure of a metric arising from a convex function on an open face of some skeleton of $X^{\textup{an}}$ is equal to the real Monge-Ampère measure of that function up to multiplication by a constant. As a consequence we obtain a regularity result for solutions of the non-archimedean Monge-Ampère problem on curves. \\[.25cm]
		MSC: Primary 32P05; Secondary 32W20, 14G22, 14T05 \\[.25cm]
		Keywords: Monge-Ampère operator, Berkovich spaces, Metrics, Tropical geometry
	\end{abstract}
	
	\maketitle
  \tableofcontents
  
	\section{Introduction}
	The non-archimedean analogue of the Calabi conjecture is still an open problem in non-archimedean geometry. In the complex case it states that for a complex compact $n$-dimensional manifold $M$ with a Kähler form $\omega$ and $f\in C^\infty(M)$, $f>0$ such that $\int_M f\omega^n=\int_M \omega^n$ there exists a unique up to constant $\varphi\in C^\infty(M)$ such that $\omega+dd^c\varphi>0$ and $(\omega+dd^c\varphi)^n=f\omega^n$. This was solved by Calabi (uniqueness, \cite{C}) and Yau (existence, \cite{Yau}). A strategy to attack the non-archimedean Monge-Ampère equation was proposed by Kontsevich and Tschinkel in unpublished though influential notes dated around 2001. In the non-archimedean setting, we fix a non-archimedean, non-trivially valued field $K$ and a smooth projective variety $X$ over $K$ of dimension $n$ with a line bundle $L$ on $X$ and consider the corresponding $K$-analytic space $X^{\textup{an}}$ with the line bundle $L^{\textup{an}}$ in the sense of Berkovich. To any continuous semipositive metric $\|\cdot\|$ on $L^{\textup{an}}$ one can associate a positive Radon measure $c_1(L,\|\cdot\|)^n$ on $X^{\textup{an}}$, called the Monge-Ampère measure, which was introduced by Chambert-Loir in \cite{Ch}. In a non-archimedean analogue of the Calabi conjecture one asks for a solution of $c_1(L,\|\cdot\|)^n=\mu$ for a positive Radon measure $\mu$ on $X^{\textup{an}}$ of mass $L^n$ when $L$ is ample. The uniqueness up to addition of a constant of such a solution was proved by Yuan and Zhang in \cite{YZ}. The existence was proved by Liu in \cite{L} for the case of a totally degenerate abelian variety $X$ under some regularity assumptions on the measure by reducing to the complex case. The best known existence result is due to Boucksom, Favre and Jonsson \cite[Theorem A]{BFJ1}. They prove existence of a solution to the non-archimedean Monge-Ampère equation if $K$ is discretely valued of residue characteristic zero and $\mu$ is supported on the dual complex of some SNC model of $X$. Note that they assumed also an algebraicity condition which was later removed by Burgos Gil, Gubler, Jell, Künnemann and Martin \cite[Theorem D]{BGJKM}. As such a dual complex consists of faces which look like simplices in $\mathbb{R}^n$ it would be tempting to observe a connection of the non-archimedean Monge-Ampère operator with the real one. This is the aim of the paper at hand. In particular we will prove the following result (a precise definition of the occurring measures is given in section \ref{measures}.):
	\begin{Theorem}
		\label{main} Let $X$ be an $n$-dimensional proper algebraic variety over $K$, $\overline{L}=(L,\|\cdot\|)$ a formally metrized line bundle on $X^{\textup{an}}$ and $\tau$ an open face of dimension $n$ of a skeleton corresponding to a strictly semistable formal model $\mathfrak{X}$ of $X^{\textup{an}}$ on which $\overline{L}$ has a formal model $\mathfrak{L}$. Let $\varphi$ be a continuous function on $X^{\textup{an}}$ such that $\|\cdot\|e^{-\varphi}$ is a semipositive metric. Suppose that $\varphi$ factorizes through the retraction $p_{\mathfrak{X}}$ onto the skeleton. Then
		\[
			c_1(L,\|\cdot\|e^{-\varphi})^n=[\tilde{K}(S):\tilde{K}]\cdot n!\cdot\MA\left(\varphi\Big|_{\tau}\right)
		\]
		on $p_{\mathfrak{X}}^{-1}(\tau)$ where $\MA$ denotes the real Monge-Ampère operator on $\tau$ which is considered to be a measure on $p_{\mathfrak{X}}^{-1}(\tau)$ by pushforward via the inclusion and $S$ denotes the point in the special fibre of $\mathfrak{X}$ which is the image of $\tau$ under the reduction map.
	\end{Theorem}
	The paper is organized as follows: In Section \ref{formalgeometry} we give an overview over basic concepts in formal geometry. We recall the definition of a strongly nondegenerate strictly polystable formal scheme and its associated skeleton introduced in \cite{Be2} and explain the stratum face correspondence developed in \cite{Gub2}. At the end of the section we construct a Cartier divisor from a piecewise affine linear function on the skeleton and prove an important lemma dealing with the degree with respect to this divisor in the case of an affine linear function. \par
	In Section \ref{Metrics} we collect basic definitions and facts on metrized line bundles. Following \cite{GM} we introduce piecewise linear, algebraic and formal metrics and the notion of semipositivity for them. We also recall some useful properties and the situations in which the definitions coincide. \par
	In Section \ref{measures} we recall the definitions of the real and non-archimedean Monge-Ampère measure but we define the latter locally on open subsets of the analytification of a separated scheme of finite type over the field $K$. In order to do so, we prove a local convergence result. This will allow us to formulate Theorem \ref{main} in a more general setting where everything is defined locally, see Corollary \ref{mainresult}. \par
	Section \ref{sectioncomparison} is subject to the proof of Theorem \ref{main}. In fact in Corollary \ref{mainresult}, we prove a local generalization of this result. It will follow from Lemma \ref{local} and Corollary \ref{semipositive->convex} that Corollary \ref{mainresult} implies Theorem \ref{main}. We will also generalize the local result in Corollary \ref{mainpolystable} to strongly nondegenerate polystable formal models of $X^{\textup{an}}$ i.e. we will prove:
	\begin{Theorem}
		Let $X$ be an $n$-dimensional proper algebraic variety over $K$ and $\mathfrak{X}$ a strongly nondegenerate polystable formal model of $X^{\textup{an}}$ over $K^\circ$ with associated skeleton $\Delta$. Let $\tau$ be an $n$-dimensional open face of $\Delta$ with associated point $S$ in the special fibre of $\mathfrak{X}$. Let $h$ be a convex function on $\tau$ and denote by $\overline{\mathcal{O}}^{h\circ p_{\mathfrak{X}}}$ the trivial line bundle on the strictly $K$-analytic space $p_{\mathfrak{X}}^{-1}(\tau)$ endowed with the metric given by $\|1\|=e^{-h\circ p_{\mathfrak{X}}}$. Then
		\[
			c_1\left(\overline{\mathcal{O}}^{h\circ p_{\mathfrak{X}}}\right)^n=[\tilde{K}(S):\tilde{K}]\cdot n!\cdot\MA(h)
		\]
		on $p_{\mathfrak{X}}^{-1}(\tau)$.
	\end{Theorem}
	The proof is inspired by the proof of \cite[Theorem 5.18]{Gub2}. In order to reduce to the toric situation, a key ingredient will be Lemma \ref{numericallyequivalent}, showing that affine linear functions on a closed face of a skeleton induce numerically trivial vertical Cartier divisors on a suitable part of the corresponding formal model. \par
	Finally, in Section \ref{application}, we apply Theorem \ref{main} to obtain two regularity results for solutions to the non-archimedean Calabi-Yau problem. For example we will prove in Proposition \ref{regularitycurves}:
	\begin{Proposition}
		\label{regularity} Let $X$ be a smooth projective curve, $\mu$ a positive Borel meausre on $X^{\textup{an}}$ and $\varphi$ a solution to the Monge-Ampère equation $c_1(L,\|\cdot\|e^{-\varphi})=\mu$. Let $\tau$ be an open face of a skeleton associated to a strictly semistable formal model of $X^{\textup{an}}$ on which $(L,\|\cdot\|)$ has a formal model. Suppose that $\mu$ is supported on that skeleton and is given on $\tau$ by $f\cdot\boldsymbol{dx}$ where $\boldsymbol{dx}$ denotes the Lebesgue measure on $\tau$. If $f\in C^k(\tau)$ then we have $\varphi\Big|_{\tau}\in C^{k+2}(\tau)$.
	\end{Proposition}
	Here $C^k(\tau)$ is the space of $k$ times continuously differentiable functions on $\tau$. Theorem \ref{regularity} follows from Theorem \ref{main} and regularity of the real Monge-Ampère equation. \\[2pt]
	\textbf{Terminology.} In the following, $K$ denotes a complete, non-archimedean, non-trivially valued field and $K^\circ$ its corresponding valuation ring with maximal ideal $K^{\circ\circ}$. All schemes are assumed to be locally of finite type. \\[2pt]
	\textbf{Acknowledgements.} I thank Walter Gubler for his constant advice and many helpful discussions. I am also grateful to Sébastien Boucksom for helpful discussions and to Antoine Ducros for suggesting a generalization of \cite[Lemme 6.5.1]{ChD}. Furthermore I would like to thank Klaus Künnemann and Antoine Chambert-Loir for helpful comments, Thomas Fenzl for answering my questions about skeletons and Florent Martin and Walter Gubler for the permission to use their unpublished notes on convexity of psh-functions.
	\section{Skeletons, formal models and divisors}
	\label{formalgeometry} In this section we first define formal schemes and their generic and special fibres. For details we refer to \cite[II.7, II.8.3]{Bo}. Then we recall the concept of skeletons associated to strongly nondegenerate strictly polystable formal schemes introduced by Berkovich in \cite{Be2}. To a subdivision of the skeleton, one can associate a formal analytic structure as in \cite[Proposition 5.5]{Gub2}. We generalize the subsequent results of \cite[§5]{Gub2} concerning the stratum face correspondence by dropping the condition of algebraically closedness of the base field. Finally we explain how a piecewise affine linear function on the skeleton induces a Cartier divisor on the formal scheme corresponding to a suitable subdivision of the skeleton.
	\begin{Def}
		Let $Y$ be a reduced scheme of locally finite type over a field $\kappa$. Set $Y^{(0)}:=Y$ and let $Y^{(i+1)}$ be the complement of the set of normal points in $Y^{(i)}$. The irreducible components of $Y^{(i)}\setminus Y^{(i+1)}$ are called \textit{strata} of $Y$. There is a partial ordering on the set of strata given by $R_1\leq R_2$ if and only if $\overline{R_1}\subseteq\overline{R_2}$. A cycle $Z\in Z(Y)$ is called a \textit{strata cycle} if there are strata $S_1,...,S_n$ of $Y$ such that $Z=\sum m_i\overline{S_i}$ with $m_i\in\mathbb{R}$. 
	\end{Def}
	\begin{Def}
		A topological ring $A$ is called \textit{adic} if there is an ideal $\mathfrak{a}\subseteq A$ such that the ideals $(\mathfrak{a}^n)_{n\in\mathbb{N}}$ form a neighbourhood basis for $0$. We call $\mathfrak{a}$ a \textit{defining ideal}. Let $A$ be an adic, complete, separated ring with finitely generated defining ideal $\mathfrak{a}$. The \textit{affine formal scheme} of $A$ is the locally topologically ringed space $\Spf(A)=(\mathfrak{X},\mathcal{O}_{\mathfrak{X}})$ where $\mathfrak{X}$ and $\mathcal{O}_{\mathfrak{X}}$ are defined as follows: $\mathfrak{X}$ is the set of all open prime ideals of $A$. As a prime ideal is open if and only if it contains $\mathfrak{a}$, we may identify $\mathfrak{X}$ with $\Spec(A/\mathfrak{a})\subseteq\Spec(A)$ and we endow $\mathfrak{X}$ with the topology induced by the Zariski topology on $\Spec(A)$. Moreover we define
		\[
			\mathcal{O}_{\mathfrak{X}}:=\underset{\leftarrow}{\lim}\;\mathcal{O}_{\Spec(A/\mathfrak{a}^n)}.
		\]
		A \textit{formal scheme} is a locally topologically ringed space $(\mathfrak{X},\mathcal{O}_{\mathfrak{X}})$ such that for each $x\in\mathfrak{X}$ there is an open neighbourhood $\mathfrak{U}$ of $x$ with $\left(\mathfrak{U},\mathcal{O}_{\mathfrak{X}}\Big|_{\mathfrak{U}}\right)$ isomorphic to an affine formal scheme.
		
		Now let $\mathfrak{a}$ be a defining ideal of $K^\circ$. A topological $K^\circ$-algebra $A$ is called \textit{admissible}, if $\left\{a\in A\;\Big|\;\mathfrak{a}^n\cdot a=0\text{ for some }n\in\mathbb{N}\right\}=\{0\}$ i.e. $A$ does not have $K^\circ$-torsion and if $A$ is isomorphic to a $K^\circ$-algebra of the form $K^\circ\langle\zeta_1,...,\zeta_n\rangle/(a_1,...,a_m)$ endowed with the $\mathfrak{a}$-adic topology. A formal $K^\circ$-scheme $\mathfrak{X}$ is called \textit{admissible} if there is a locally finite open cover $(\mathfrak{U}_i)_{i\in I}$ of $\mathfrak{X}$ with $\mathfrak{U}_i=\Spf(A_i)$ for admissible $K^\circ$-algebras $A_i$.
		
		Let $\mathfrak{X}=\Spf(A)$ be an admissible formal affine $K^\circ$-scheme. The \textit{analytic generic fibre} of $\mathfrak{X}$ is defined as $\mathfrak{X}^{\textup{an}}:=\mathcal{M}(A\otimes_{K^\circ} K)$, where $\mathcal{M}(\cdot)$ denotes the Berkovich spectrum (cf. \cite[1.2]{Be1}). The special fibre of $\mathfrak{X}$ is given by $\tilde{\mathfrak{X}}:=\Spec(A\otimes_{K^\circ} k)$, where $k:=K^\circ/K^{\circ\circ}$ is the residue field of $K$. For an admissible formal $K^\circ$-scheme $\mathfrak{X}$ one obtains the generic and the special fibre by a gluing process. There is a canonical surjective reduction map $\red:\mathfrak{X}^{\textup{an}}\rightarrow\tilde{\mathfrak{X}}$, see \cite[§2.13]{GRW}.
	\end{Def}
	\begin{Def}
		\label{Defstrictlysemistable} For $n\in\mathbb{N}_{>0}$ and $a\in K^{\circ\circ}$ we define 	
		\[
			\mathfrak{X}(n,a):=\Spf(K^\circ\langle x_0,...,x_n\rangle/(x_0...x_n-a)).
		\]
		For tuples $\boldsymbol{n}=(n_0,...,n_p)\in\mathbb{N}_{>0}^{p+1}$ and $\boldsymbol{a}=(a_0,...,a_p)\in (K^{\circ\circ})^{p+1}$ we define $\mathfrak{X}(\boldsymbol{n},\boldsymbol{a}):=\mathfrak{X}(n_0,a_0)\times_{K^{\circ}}...\times_{K^{\circ}}\mathfrak{X}(n_p,a_p)$ and for $m\in\mathbb{N}$ we set $\mathfrak{X}(m):=\mathfrak{X}(m,1)$. A \textit{strictly polystable} formal scheme over $K^\circ$ is an admissible formal scheme $\mathfrak{X}$ over $K^\circ$ which can be covered by formal open sets $\mathfrak{U}$ with étale morphisms 
		\[	\psi:\mathfrak{U}\rightarrow\mathfrak{X}(\boldsymbol{n},\boldsymbol{a},m):=\mathfrak{X}(\boldsymbol{n},\boldsymbol{a})\times_{K^\circ}\mathfrak{X}(m)
		\]
		where $\boldsymbol{n}$, $\boldsymbol{a}$ and $m$ may depend on $\mathfrak{U}$. We say that $\mathfrak{X}$ is \textit{strongly nondegenerate strictly polystable} if all $a_i$ can be chosen nonzero.
		
		To a strongly nondegenerate strictly polystable formal scheme $\mathfrak{X}$ over $K^\circ$ Berkovich introduced in \cite{Be2} a canonical polytopal subset $S(\mathfrak{X})$ of $\mathfrak{X}^{\textup{an}}$ called the \textit{skeleton}. It is a closed subset of $\mathfrak{X}^{\textup{an}}$ which is locally given by canonical polysimplices and can be described as follows. Let $\psi:\mathfrak{U}\rightarrow\mathfrak{X}(\boldsymbol{n},\boldsymbol{a},m)$ be an étale morphism as above. The generic fibre of the right hand side is given as $\mathfrak{X}(\boldsymbol{n},\boldsymbol{a},m)^{\textup{an}}=\mathscr{M}(A)$ where $A=(K\langle T_0^{\pm},...,T_m^{\pm}\rangle)\langle T_{00},...,T_{p,n_p}\rangle/(T_{00}...T_{0,n_0}-a_0,...,T_{p0}...T_{p,n_p}-a_p)$. The elements of $A$ can be expressed as $\sum_\mu a_\mu T^\mu$ with $a_\mu\in K\langle T_0^{\pm},...,T_m^{\pm}\rangle$ and $a_\mu=0$ if there is an $i\in\{0,...,p\}$ such that $\mu_{i,k}\geq 1$ for all $k\in\{0,...,n_i\}$. Now to an element $\boldsymbol{t}$ in the polysimplex $\left\{\boldsymbol{t}\in\mathbb{R}_{\geq 0}^{\boldsymbol{n}+\boldsymbol{1}}\;\Big|\;t_{i0}+...+t_{in_i}=-\log(|a_i|), 0\leq i\leq p\right\}$ we associate a seminorm on $A$ by sending a power series as above to $\max_\mu\{|a_\mu|\exp(-\boldsymbol{t}\cdot\mu)\}$. This gives an embedding of the polysimplex into $\mathscr{M}(A)$ whose image is denoted by $\Delta$. The skeleton $S(\mathfrak{U})$ of $\mathfrak{U}$ is defined to be $(\psi^{\textup{an}})^{-1}(\Delta)$. One can show that $\psi^{\textup{an}}$ induces a homeomorphism from $(\psi^{\textup{an}})^{-1}(\Delta)$ to $\Delta$ if $\mathfrak{U}$ has a unique minimal stratum which maps to the minimal stratum of $\mathfrak{X}(\boldsymbol{n},\boldsymbol{a},m)$. The skeleton $S(\mathfrak{X})$ of $\mathfrak{X}$ is the union of all $S(\mathfrak{U})$ and is independent of all choices.
		
		To a stratum $S$ of $\mathfrak{X}$ one can associate a canonical polysimplex $\Delta_S$ in the skeleton such that the interiors of the $\Delta_T$ form a disjoint cover of $S(\mathfrak{X})$ where $T$ ranges over all strata of $\tilde{\mathfrak{X}}$. In order to do so, we choose a refinement of the cover of $\mathfrak{X}$ as described in the Proposition below and choose $\mathfrak{U}$ such that $S$ is its distinguished stratum. We then define $\Delta_S:=S(\mathfrak{U})$.
		
		An admissible formal scheme $\mathfrak{X}$ is called \textit{strongly nondegenerate polystable} if there exists a strongly nondegenerate strictly polystable formal scheme $\mathfrak{X}'$ and a surjective étale morphism $\mathfrak{X}'\rightarrow\mathfrak{X}$. The skeleton of $\mathfrak{X}$ is defined to be the image of the skeleton of $\mathfrak{X}'$ under the map $\mathfrak{X}'^{\textup{an}}\rightarrow\mathfrak{X}^{\textup{an}}$.
		
		One can endow the skeleton with a piecewise linear structure, see \cite[§6]{Be3}. We will define piecewise affine linear functions on the skeleton of a strongly nondegenerate strictly polystable formal scheme in Definition \ref{Defpiecewiseaffinelinear}.
		There is a canonical continuous retraction map $p_{\mathfrak{X}}:\mathfrak{X}^{\textup{an}}\rightarrow S(\mathfrak{X})$ which restricts to the identity on $S(\mathfrak{X})$. For details see \cite[§4]{Be2}, \cite[§4]{Be3} or \cite[5.3]{Gub2}.
	\end{Def}
	We have the following stratum face correspondence due to Berkovich:
	\begin{Proposition}
		\label{SFCBer} Let $\mathfrak{X}$ be a strongly nondegenerate polystable formal scheme with skeleton $\Delta$. There is a bijective correspondence between the open faces of $\Delta$ and the strata of $\tilde{\mathfrak{X}}$ given by
		\[
		  R=\red(p_{\mathfrak{X}}^{-1}(\tau)), \hspace{2cm} \tau=p_{\mathfrak{X}}(\red^{-1}(R)).
		\]
	\end{Proposition}
	\begin{proof}
		\cite[Theorem 5.2 (iv), Theorem 5.4]{Be2}.
	\end{proof}
	\begin{Proposition}
		\label{Gub2 Proposition 5.2} Let $\mathfrak{X}$ be a strongly nondegenerate strictly polystable formal scheme over $K^\circ$. Any formal open covering of $\mathfrak{X}$ admits a refinement $\{\mathfrak{U}'\}$ by formal open subsets $\mathfrak{U}'$ as in Definition \ref{Defstrictlysemistable} such that
		\begin{enumerate}
			\item Every $\mathfrak{U}'$ is a formal affine open subscheme of $\mathfrak{X}$,
			\item there is a distinguished stratum $S$ of $\tilde{\mathfrak{X}}$ associated to $\mathfrak{U}'$ such that for any stratum $T$ of $\tilde{\mathfrak{X}}$, we have $S\subseteq\overline{T}$ if and only if $\tilde{\mathfrak{U}'}\cap\overline{T}\neq\emptyset$,
			\item $\tilde{\psi}^{-1}(\{\tilde{\boldsymbol{0}}\}\times\widetilde{\mathfrak{X}(m)})$ is the stratum of $\tilde{\mathfrak{U}'}$ which is equal to $\tilde{\mathfrak{U}'}\cap S$ for the distinguished stratum $S$ associated to $\mathfrak{U}'$,
			\item every stratum of $\tilde{\mathfrak{X}}$ is the distinguished stratum of a suitable $\mathfrak{U}'$.
		\end{enumerate}
	\end{Proposition}
	\begin{proof}
		The very same arguments as in \cite[Proposition 5.2]{Gub2} apply to our situation.
	\end{proof}
	From now on let $\mathfrak{X}$ be a strongly nondegenerate strictly polystable formal scheme over $K^\circ$ and denote by $\Gamma$ the value group of $K$. For the basic notions of convex geometry we refer to \cite[Appendix A]{Gub3}. We will work with $\Gamma$-rational \textit{polytopal subdivisions} $\mathfrak{D}$ of $S(\mathfrak{X})$, i.e. $\mathfrak{D}$ is a family of $\Gamma$-rational polytopes contained in a canonical polysimplex such that for every stratum $S$ of $\tilde{\mathfrak{X}}$ the set $\left\{\Delta\in\mathfrak{D}\;\Big|\;\Delta\subseteq\Delta_S\right\}$ is a polytopal decomposition of $\Delta_S$. Here a \textit{polytopal decomposition} means a finite family of polytopes covering $\Delta_S$ which is closed under taking faces and such that the intersection of two polytopes in the family is a face of both and a $\Gamma$\textit{-rational polytope} means a polytope which is defined by inequalities of the form $\boldsymbol{m}\boldsymbol{x}+c\geq 0$ with $\boldsymbol{m}\in\mathbb{Z}^r,c\in\Gamma$.
	\begin{Construction}
		\label{Formalmodel} Let $\mathfrak{D}$ be such a subdivision. We will construct a canonical formal scheme $\mathfrak{X}''$ over $K^\circ$ associated to $\mathfrak{D}$ together with a morphism $\iota:\mathfrak{X}''\rightarrow\mathfrak{X}$ which induces the identity on the generic fibre such that there is a one to one correspondence between the open faces of $\mathfrak{D}$ and the strata of $\tilde{\mathfrak{X}}''$. First of all we choose a covering of $\mathfrak{X}$ as in Proposition \ref{Gub2 Proposition 5.2}. Let $\mathfrak{U}$ be a member of this covering with an étale morphism $\psi:\mathfrak{U}\rightarrow\mathfrak{X}(\boldsymbol{n},\boldsymbol{a},m)$ and let $S$ be the distinguished stratum of $\mathfrak{U}$. For $\Delta\in\mathfrak{D}\cap\Delta_S$ we set 
		\[
			A':=\left\{\sum_\mu a_\mu T^\mu\in K((T_{00},...,T_{p,n_p}))\;\Big|\;\forall_{u\in\Delta}:\;\lim v(a_\mu)+u\cdot\mu=\infty\right\}
		\]
		and $A:=A'/(T_{00}...T_{0,n_0}-a_0,...,T_{p0}...T_{p,n_p}-a_p)$ and define 
		\[
			A^\Delta:=\left\{\sum_\mu a_\mu T^\mu\in A\;\Big|\;\forall_{u\in\Delta,\mu\in\mathbb{Z}^{\boldsymbol{n}+\boldsymbol{1}}}:\;v(a_\mu)+\mu\cdot u\geq 0\right\}
		\] 
		and $\mathfrak{U}_\Delta:=\Spf A^\Delta$. If $\Delta_1,\Delta_2\in\mathfrak{D}\cap\Delta_S$ then $\Delta_1\cap\Delta_2$ is a face of both and by transferring the arguments in \cite[Proposition 6.12]{Gub3} to the analytic situation, we obtain that the canonical morphisms $\mathfrak{U}_{\Delta_1\cap\Delta_2}\rightarrow\mathfrak{U}_{\Delta_i}$ are open immersions. Hence we can glue the $\mathfrak{U}_{\Delta}$ along this data to obtain a formal scheme which we denote by $\mathfrak{X}(\boldsymbol{n},\boldsymbol{a})'$ together with a morphism $\iota':\mathfrak{X}(\boldsymbol{n},\boldsymbol{a})'\rightarrow\mathfrak{X}(\boldsymbol{n},\boldsymbol{a})$. Let $\psi':\mathfrak{U}''\rightarrow\mathfrak{X}(\boldsymbol{n},\boldsymbol{a})'\times\mathfrak{X}(m)$ be the base change of $\psi$ with respect to $\iota'\times\Id$. The construction of $\mathfrak{U}''$ does not depend on the choice of $\psi$ up to isomorphism: Let $\rho:\mathfrak{U}\rightarrow\mathfrak{X}(\boldsymbol{n},\boldsymbol{a},m)$ be another étale morphism. Then up to reordering the coordinates, $\rho^{\ast}x_i=u_i\psi^{\ast}x_i$ for some $u_i\in\mathcal{O}(\mathfrak{U})^\times$. Then we have canonical $K^\circ$-algebra isomorphisms:
		\begin{align*}
			\mathcal{O}(\mathfrak{U})\hat{\otimes}_{\psi^{\ast}}A^\Delta&\rightarrow\mathcal{O}(\mathfrak{U})\hat{\otimes}_{\rho^{\ast}}A^\Delta, \\
			a\otimes x_i &\mapsto u_i a\otimes x_i,	
		\end{align*}
		which yield an isomorphism of the $\mathfrak{U}''$ constructed with $\psi$ respectively $\rho$. \\
		We glue the $\mathfrak{U}''$ to obtain our formal scheme $\mathfrak{X}''$. Although $\mathfrak{X}''$ might not be admissible, we can define its generic fibre and reduction map in the usual way as the algebras $A^\Delta\otimes_{K^\circ}K$ are strictly $K$-affinoid (see \cite[Proposition 6.17]{Gub3}). Then $\iota$ induces the identity on the generic fibres and we set $p_{\mathfrak{X}''}:=p_{\mathfrak{X}}$. Note that $\mathfrak{X}''$ is admissible if the vertices of the polytopes in $\mathfrak{D}$ are $\Gamma$-rational, in particular the base change of $\mathfrak{X}''$ to the valuation ring of the completion of an algebraic closure of $K$ is admissible, see \cite[Proposition 6.7]{Gub3}.
	\end{Construction}
	\begin{Remark}
		If $\mathfrak{D}$ is trivial i.e. $\Delta\in\mathfrak{D}$ only if $\Delta=\Delta_S$ for some stratum $S$ of $\tilde{\mathfrak{X}}$ then it is an immediate consequence from the construction that $\mathfrak{X}''=\mathfrak{X}$. %In particular for a closed face $\Delta$ of $S(\mathfrak{X})$ the set $p_{\mathfrak{X}}^{-1}(\Delta)$ is open in the formal topology on $X^{an}$ and hence the generic fibre of a formal open subset of $\mathfrak{X}$. This gives a strongly nondegenerate strictly polystable formal model $\mathfrak{X}'$ of $p_{\mathfrak{X}}^{-1}(\Delta)$ with associated skeleton $S(\mathfrak{X}')=\Delta$. 
	\end{Remark}
	We will frequently use the following generalization of \cite[Proposition 5.7]{Gub2} which is a stratum face correspondence for the $\mathfrak{X}''$ constructed above.
	\begin{Proposition}
		\label{Stratumface} Let $\mathfrak{X}$ be a strongly nondegenerate strictly polystable formal scheme with skeleton $\Delta$ and $\mathfrak{D}$ a subdivision of $\Delta$ with associated formal structure $\mathfrak{X}''$. Then there is a bijective correspondence between the open faces of $\mathfrak{D}$ and the strata of $\tilde{\mathfrak{X}}''$ given by 
		\[
			R=\red(p_{\mathfrak{X}''}^{-1}(\tau)), \hspace{2cm} \tau=p_{\mathfrak{X}''}(\red^{-1}(R)).
		\]
		Furthermore, in the second equality, $R$ can be replaced by any nonempty subset of $R$. 
	\end{Proposition}
	\begin{proof}
		We follow the proof of \cite[Proposition 5.7]{Gub2} but in order to establish the result for an arbitrary non-archimedean field $K$ (not necessarily algebraically closed), we use \cite[Proposition 6.22]{Gub3} instead of \cite[Proposition 4.4]{Gub1}. Let $\tau$ be an open face of $\mathfrak{D}$. We prove first that $R:=\red(p_{\mathfrak{X}''}^{-1}(\tau))$ is a stratum of $\tilde{\mathfrak{X}}''$. There is a unique stratum $S$ of $\tilde{\mathfrak{X}}$ such that $\tau$ is contained in the interior of $\Delta_S$. Let $\mathfrak{U}$ be a formal open subset of $\mathfrak{X}$ such that $S$ is the distinguished stratum of $\mathfrak{U}$ (Proposition \ref{Gub2 Proposition 5.2}). As strata are compatible with localization we may assume $\mathfrak{X}=\mathfrak{U}$. Let $\psi_1':\mathfrak{X}''\rightarrow\mathfrak{X}(\boldsymbol{n},\boldsymbol{a})'$ be the base change of the composition of the étale map $\psi:\mathfrak{X}\rightarrow\mathfrak{X}(\boldsymbol{n},\boldsymbol{a},m)$ with the projection on the first factor $\mathfrak{X}(\boldsymbol{n},\boldsymbol{a})$. By \cite[Proposition 6.22]{Gub3} the first part of the proposition holds for $\mathfrak{X}(\boldsymbol{n},\boldsymbol{a})'$. Let $T$ be the stratum of $\mathfrak{X}(\boldsymbol{n},\boldsymbol{a})'$ corresponding to $\tau$, i.e.
		\begin{align}
			\label{tau=} \tau=p_{\mathfrak{X}(\boldsymbol{n},\boldsymbol{a})'}(\red^{-1}(T))
		\end{align}
		and
		\begin{align}
			\label{T=} T=\red(p_{\mathfrak{X}(\boldsymbol{n},\boldsymbol{a})'}^{-1}(\tau)).
		\end{align}
		where $p_{\mathfrak{X}(\boldsymbol{n},\boldsymbol{a})'}:\mathfrak{X}(\boldsymbol{n},\boldsymbol{a})^{\prime an}\rightarrow\Delta_S$ is the retraction map. We prove $R=\tilde{\psi}_1^{\prime -1}(T)$. First we observe that 
		\[
			\red(({\psi'}_1^{\textup{an}})^{-1}(p_{\mathfrak{X}(\boldsymbol{n},\boldsymbol{a})'}^{-1}(\tau)))=\tilde{\psi}_1^{\prime -1}(\red(p_{\mathfrak{X}(\boldsymbol{n},\boldsymbol{a})'}^{-1}(\tau))).
		\]
		The inclusion $\subseteq$ is clear because $\red\circ{\psi'}_1^{\textup{an}}=\tilde{\psi}'_1\circ\red$. The other inclusion follows from this fact and an application of \cite[Proposition 6.22]{Gub3}. For details we refer to the proof of \cite[Proposition 5.7]{Gub2}. We conclude
		\[
			R=\red(p_{\mathfrak{X}''}^{-1}(\tau))=\red(({\psi'}_1^{\textup{an}})^{-1}(p_{\mathfrak{X}(\boldsymbol{n},\boldsymbol{a})'}^{-1}(\tau)))=\tilde{\psi}_1^{\prime -1}(\red(p_{\mathfrak{X}(\boldsymbol{n},\boldsymbol{a})'}^{-1}(\tau)))\overset{(\ref{T=})}{=}\tilde{\psi}_1^{\prime -1}(T).
		\]
		By \cite[Lemma 2.2]{Be2} $R$ is a strata subset. To see that $R$ is indeed a stratum it is enough to show that $R$ is irreducible. But this follows from
		\[
			\tilde{\psi}_1^{\prime -1}(T)=(T\times\tilde{\mathfrak{X}}(m))\times_{\tilde{\mathfrak{X}}(\boldsymbol{n},\boldsymbol{a},m)'}\tilde{\mathfrak{X}}''\cong (T\times\tilde{\mathfrak{X}}(m))\times_{\{\tilde{0}\}\times\tilde{\mathfrak{X}}(m)}\tilde{\psi}^{-1}(\{\tilde{0}\}\times\tilde{\mathfrak{X}}(m))\cong T\times S,
		\]
		where the latter is irreducible by \cite[Corollaire 4.5.8 (i)]{EGAIV2}.
		As the open faces of $\mathfrak{D}$ cover $\Delta$, every stratum of $\tilde{\mathfrak{X}}''$ is obtained this way. It remains to prove that we can recover $\tau$ from $R$. First note that 
		\[
			p_{\mathfrak{X}''}(({\psi'}_1^{\textup{an}})^{-1}(\red^{-1}(T)))=p_{\mathfrak{X}(\boldsymbol{n},\boldsymbol{a})'}(\red^{-1}(T)).
		\]
		The inclusion $\subseteq$ is clear because $p_{\mathfrak{X}''}=p_{\mathfrak{X}(\boldsymbol{n},\boldsymbol{a})'}\circ{\psi'}_1^{\textup{an}}$. For the other inclusion, let $x\in p_{\mathfrak{X}(\boldsymbol{n},\boldsymbol{a})'}(\red^{-1}(T))=\tau$. As the sets $\red^{-1}(T')$ with $T'$ varying over the strata of $\tilde{\mathfrak{X}}(\boldsymbol{n},\boldsymbol{a})'$ cover ${\mathfrak{X}(\boldsymbol{n},\boldsymbol{a})'}^{\textup{an}}$ and using \cite[Proposition 6.22]{Gub3} and the fact the $p_{\mathfrak{X}(\boldsymbol{n},\boldsymbol{a})'}$ restricts to the identity on $\Delta$ we deduce $x\in\red^{-1}(T)$. Hence $x$ is an element of the left hand side which proves the equality claimed in the display. Now the rest is an easy calculation:
		\begin{align*}
			p_{\mathfrak{X}''}(\red^{-1}(R))&=p_{\mathfrak{X}''}(\red^{-1}(\tilde{\psi}_1^{\prime -1}(T))) \\
			&=p_{\mathfrak{X}''}(({\psi'}_1^{\textup{an}})^{-1}(\red^{-1}(T))) \\
			&=p_{\mathfrak{X}(\boldsymbol{n},\boldsymbol{a})'}(\red^{-1}(T)) \\
			&\overset{(\ref{tau=})}{=}\tau.
		\end{align*}
		Finally we want to show that $R$ may be replaced by a nonempty subset $Y$ of $R$. Clearly, the arguments in \cite[Proposition 5.7]{Gub2} generalize to the polystable situation, so we presume the claim for $K$ algebraically closed and show how to drop this assumption. Let $\mathbb{C}_K$ be the completion of an algebraic closure of $K$. We denote by $\pi:\mathfrak{X}''_{\mathbb{C}_K}\rightarrow\mathfrak{X}''$ the base change of $\mathfrak{X}''$ to $\mathbb{C}_K^\circ$. Let $R'$ be the union of the strata of $\tilde{\mathfrak{X}}''_{\mathbb{C}_K}$ lying over $R$. Then $\pi$ induces a surjection $p_{\mathfrak{X}''_{\mathbb{C}_K}}(\red^{-1}(R'))\twoheadrightarrow p_{\mathfrak{X}''}(\red^{-1}(R))$ as the strata in $R'$ correspond to open faces lying over $\tau$. Let $Y'$ be a lift of $Y$ in $R'$. By \cite[Proposition 5.7]{Gub2} we have $p_{\mathfrak{X}''_{\mathbb{C}_K}}(\red^{-1}(Y'))=p_{\mathfrak{X}''_{\mathbb{C}_K}}(\red^{-1}(R'))$. Clearly $p_{\mathfrak{X}''}(\red^{-1}(Y))\subseteq p_{\mathfrak{X}''}(\red^{-1}(R))$ and hence it is enough to show that the restriction of $\pi$ to $p_{\mathfrak{X}''_{\mathbb{C}_K}}(\red^{-1}(Y'))$ factors through $p_{\mathfrak{X}''}(\red^{-1}(Y))$. We have the following commutative diagram:
		\[
			\xymatrix{
				S(\mathfrak{X}''_{\mathbb{C}_K}) \ar@{->>}[d]_\pi & p_{\mathfrak{X}''_{\mathbb{C}_K}}(\red^{-1}(Y')) \ar@{_{(}->}[l] & \red^{-1}(Y') \ar@{->>}[l]_/-0.3cm/{p_{\mathfrak{X}''_{\mathbb{C}_K}}} \ar@{.>}[d]_\pi \ar[r]^/0.2cm/\red & Y' \ar@{->>}[d]_\pi \\
				S(\mathfrak{X}'') & & \red^{-1}(Y) \ar[ll]_{p_{\mathfrak{X}''}} \ar@{->>}[r]^/0.2cm/\red & Y
			}
		\]
		Let $x\in p_{\mathfrak{X}''_{\mathbb{C}_K}}(\red^{-1}(Y'))$ and $y \in\red^{-1}(Y')$ with $p_{\mathfrak{X}''_{\mathbb{C}_K}}(y)=x$ then $\pi(x)=\pi(p_{\mathfrak{X}''_{\mathbb{C}_K}}(y))=p_{\mathfrak{X}''}(\pi(y))\in p_{\mathfrak{X}''}(\red^{-1}(Y))$. This proves the claim.
	\end{proof}
	\begin{Corollary}
		\label{Corstratumface} Let $R$ be a stratum of $\tilde{\mathfrak{X}}''$ corresponding to the open face $\tau$ of $\mathfrak{D}$.
		\begin{enumerate}[label=(\alph*)]
			\item $\dim(\tau)=\codim(R,\tilde{\mathfrak{X}}'')$.
			\item $S:=\tilde{\iota}(R)$ is a stratum of $\tilde{\mathfrak{X}}$.
			\item $R\overset{\tilde{\iota}}{\rightarrow} S$ is a fibre bundle with fibre $T$ where $T$ is the $\dim(R)-\dim(S)$ dimensional torus orbit from the proof of Proposition \ref{Stratumface}.
			\item Every stratum of $\tilde{\mathfrak{X}}''$ is smooth.
			\item The closure $\bar{R}$ is the union of all strata of $\tilde{\mathfrak{X}}''$ corresponding to open faces $\sigma$ of $\mathfrak{D}$ with $\tau\subseteq\bar{\sigma}$.
			\item For an irreducible component $Y$ of $\tilde{\mathfrak{X}}''$, let $\zeta_Y$ be the unique point of $\mathfrak{X}^{\textup{an}}$ with reduction equal to the generic point of $Y$. Then $Y\mapsto\zeta_Y$ is a bijection between the irreducible components of $\tilde{\mathfrak{X}}''$ and the vertices of $\mathfrak{D}$. 
		\end{enumerate}
	\end{Corollary}
	\begin{proof}
		The statements can be proven the same way as in \cite[Corollary 5.9]{Gub2}. In order to bypass the algebraically closedness of the base field one can use \cite[Proposition 6.22]{Gub3} instead of \cite[Proposition 4.4]{Gub1} for (a), \cite[Proposition 6.22]{Gub3} instead of \cite[Remark 4.8]{Gub1} for (e) and \cite[Proposition 6.14]{Gub3} instead of \cite[Proposition 4.7]{Gub1} for (f).
	\end{proof}
	\begin{Def}
		\label{Defpiecewiseaffinelinear} Let $\Delta$ be a skeleton associated to a strongly nondegenerate strictly polystable formal scheme $\mathfrak{X}'$ over $K^{\circ}$. A continuous function $h:\Delta\rightarrow\mathbb{R}$ is called \textit{piecewise affine linear} if there exists a $\Gamma$-rational polytopal subdivision $\mathfrak{D}$ of $\Delta$ such that for any canonical polysimplex $\Delta_{S}$ of $\Delta$, any formal open subset $\psi:\mathfrak{U}\rightarrow\mathfrak{X}(\boldsymbol{n},\boldsymbol{a},m)$ of $\mathfrak{X}'$ whose distinguished stratum is $S$ and any $\Delta'\in\mathfrak{D}$ with $\Delta'\subseteq\Delta_S$, there exist $\boldsymbol{m}\in\mathbb{Z}^{\boldsymbol{n}+\boldsymbol{1}}$ and $\alpha\in K^\times$ such that $h\Big|_{\Delta'}=(\boldsymbol{m}\cdot\boldsymbol{x}+v(\alpha))\circ\psi^{\textup{an}}\Big|_{\Delta'}$ (see Definition \ref{Defstrictlysemistable} for the notation and setting).
	\end{Def}
	\begin{Proposition}
		\label{CartierDiv}
		Let $\mathfrak{X}'$ be a strongly nondegenerate strictly polystable formal scheme over $K^{\circ}$ with associated skeleton $S(\mathfrak{X}')$ and $h$ a piecewise affine linear function on $S(\mathfrak{X}')$. Let $\mathfrak{D}$ be a $\Gamma$-rational polytopal subdivision of $S(\mathfrak{X}')$ suitable for $h$ as in Definition \ref{Defpiecewiseaffinelinear} and $\iota:\mathfrak{X}''\rightarrow\mathfrak{X}'$ be the canonical formal scheme over $\mathfrak{X}'$ associated to $\mathfrak{D}$ (see Construction \ref{Formalmodel}). Then $h$ induces a canonical Cartier divisor $D$ on $\mathfrak{X}''$ which is trivial on the generic fibre. If $\mathfrak{X}''$ is admissible, then $D$ has the property that $\|1\|_{\mathcal{O}(D)}=e^{-h\circ p_{\mathfrak{X}'}}$ where $\|\cdot\|_{\mathcal{O}(D)}$ is the formal metric on $\mathcal{O}_{\mathfrak{X}'^{\textup{an}}}$ given by the formal model $\mathcal{O}(D)$ of $\mathcal{O}_{\mathfrak{X}'^{\textup{an}}}$ (see Definition \ref{metrics}).
	\end{Proposition}
	\begin{proof}
		As in Construction \ref{Formalmodel}, we cover $\mathfrak{X}'$ by étale maps $\psi:\mathfrak{U}\rightarrow\mathfrak{X}(\boldsymbol{n},\boldsymbol{a},m)=\mathfrak{X}(\boldsymbol{n},\boldsymbol{a})\times\mathfrak{X}(m)$ and for each $\mathfrak{U}$ and $\Delta\in\mathfrak{D}$ with $\Delta\subseteq\mathfrak{U}^{\textup{an}}$ we obtain the affine formal scheme $\mathfrak{U}_\Delta$. We write $\psi':\mathfrak{U}''_\Delta\rightarrow\mathfrak{U}_\Delta$ for the base change with respect to $\psi$ and obtain a cover of $\mathfrak{X}''$. On $\Delta\in\mathfrak{D}$, $h$ is given by $\boldsymbol{m}\boldsymbol{x}+v(\alpha)$ with $\boldsymbol{m}\in\mathbb{Z}^{\boldsymbol{n}+\boldsymbol{1}}$, $\alpha\in K^\times$. We define $D$ locally on $\mathfrak{U}''_{\Delta'}$ by $\psi^{\prime\ast}(\alpha\cdot\boldsymbol{x}^{\boldsymbol{m}})$. Then $D$ is indeed a Cartier Divisor on $\mathfrak{X}''$ as for $\mathfrak{U}_1,\mathfrak{U}_2,\Delta_1,\Delta_2$ as above and $\mathfrak{U}:=\mathfrak{U}_1\cap\mathfrak{U}_2$ we have $\alpha_1 \cdot\boldsymbol{x}^{\boldsymbol{m}_1}/\alpha_2 \cdot\boldsymbol{x}^{\boldsymbol{m}_2}\in\mathcal{O}(\mathfrak{U}_{\Delta_1\cap\Delta_2})^\times$ since $\boldsymbol{m}_1\boldsymbol{x}+v(\alpha_1)=\boldsymbol{m}_2\boldsymbol{x}+v(\alpha_2)$ on $\Delta_1\cap\Delta_2$. Hence
\begin{align*}						   
	\psi_1^{\prime\ast}(\alpha_1\cdot\boldsymbol{x}^{\boldsymbol{m}_1})/\psi_2^{\prime\ast}(\alpha_2\cdot\boldsymbol{x}^{\boldsymbol{m}_2})\Big|_{\mathfrak{U}''_{\Delta_1\cap\Delta_2}}=\psi^{\prime\ast}(\alpha_1\cdot\boldsymbol{x}^{\boldsymbol{m}_1}/\alpha_2\cdot\boldsymbol{x}^{\boldsymbol{m}_2})\in\mathcal{O}(\mathfrak{U}''_{\Delta_1\cap\Delta_2})^\times
\end{align*}
and therefore $\psi_1^{\prime\ast}(\alpha_1\boldsymbol{x}^{\boldsymbol{m}_1})/\psi_2^{\prime\ast}(\alpha_2\boldsymbol{x}^{\boldsymbol{m}_2})\in\mathcal{O}(\mathfrak{U}''_{1,\Delta_1}\cap\mathfrak{U}''_{2,\Delta_2})^\times$.
Furthermore $D$ is trivial on the generic fibre, as $\alpha\cdot\boldsymbol{x}^{\boldsymbol{m}}\in\mathcal{O}(\mathfrak{X}(\boldsymbol{n},\boldsymbol{a})^{\textup{an}})^\times$.
	\end{proof}
	\begin{Remark}
		Note that we can ensure that $\mathfrak{X}''$ is admissible and hence a formal model by performing base change to the completion of an algebraic closure of $K$ (see Construction \ref{Formalmodel}) which will be enough for our purposes.
	\end{Remark}
	\begin{Lemma}
		\label{numericallyequivalent} In the situation of Proposition \ref{CartierDiv} let $\tau$ be an open face of the skeleton $\Delta$ of dimension equal to the dimension of $\mathfrak{X}^{\prime an}$ and assume that $h$ is affine linear on $\bar{\tau}$. Let $D$ be the induced Cartier divisor on $\mathfrak{X}''$ and $Y$ a proper curve in $\tilde{\mathfrak{X}}''$ with $Y\subseteq\red_{\mathfrak{X}''}(p_{\mathfrak{X}''}^{-1}(\tau))$ e.g. if $Y$ lies inside an irreducible component of $\tilde{\mathfrak{X}}''$ corresponding to a vertex $u\in\tau$ of $\mathfrak{D}$. Then $\Deg(D.Y)=0$.
	\end{Lemma}
	\begin{proof}
	Note that we do not assume $\bar{\tau}\in\mathfrak{D}$. But by passing to the formal open subscheme of $\mathfrak{X}'$ consisting of the formal open subsets $\mathfrak{U}$ with $S(\mathfrak{U})=\bar{\tau}$, we may assume $\Delta=\bar{\tau}$ and then the polytopal subdivision $\mathfrak{D}'$ consisting the polytope $\bar{\tau}$ and its faces is suitable for $h$. The corresponding formal scheme is $\mathfrak{X}'$. Let $D'$ be the Cartier divisor on $\mathfrak{X}'$ induced by $h$ as in Proposition \ref{CartierDiv}. Notice that by construction we have $D=\iota^{\ast}D'$. Now $\iota$ is proper by \cite[Corollary 4.4]{T1} (the result requires $\mathfrak{X}''$ to be admissible but by \cite[Proposition 2.7.1]{EGAIV2} it is enough to check properness after base change to the completion of an algebraic closure of $K$, after which $\mathfrak{X}''$ is always admissible, see Construction \ref{Formalmodel}). Hence the projection formula yields $\Deg(D.Y)=\Deg(D'.\iota_{\ast}Y)$. Now
	\[
\iota(Y)\subseteq\iota(\red_{\mathfrak{X}''}(p_{\mathfrak{X}''}^{-1}(\tau)))=\red_{\mathfrak{X}'}(p_{\mathfrak{X}'}^{-1}(\tau)),
\]
where the latter is the stratum in $\tilde{\mathfrak{X}}'$ corresponding to $\tau$ and hence a point. Therefore $D'.\iota_{\ast}Y=0$.
	\end{proof}
	\section{Metrics}
	\label{Metrics} In this section we introduce metrics on line bundles on strictly $K$-analytic spaces. This includes piecewise linear, algebraic and formal metrics. We will see that under certain conditions they are all the same. The main reference is \cite{GM}. 
	\begin{Def}
		\label{metrics} Let $X$ be a strictly $K$-analytic space and $L$ a line bundle on $X$, i.e. a locally free sheaf of rank 1 on the G-topology. A \textit{continuous metric} $\|\cdot\|$ on $L$ is a function which asserts to any admissible open subset $U\subseteq X$ and any section $s\in\Gamma(U,L)$ a continuous (with respect to the Berkovich topology) function $\|s(\cdot)\|:U\rightarrow\mathbb{R}_{\geq 0}$ such that:
		\begin{enumerate}
			\item For an admissible open subset $V\subseteq U$ we have $\left\|s\Big|_V(\cdot)\right\|=\|s(\cdot)\|\Big|_V$,
			\item for $f\in\Gamma(U,\mathcal{O}_{X})$ we have $\|fs(\cdot)\|=|f(\cdot)|\|s(\cdot)\|$,
			\item for $p\in U$ we have $\|s(p)\|=0$ if and only if $s(p)=0$.
		\end{enumerate}
		Given a formal model $(\mathfrak{X},\mathfrak{L})$ of $(X,L)$ one can define an associated so called \textit{formal metric} $\|\cdot\|_{\mathfrak{L}}$ on $L$ in the following way: If $s$ is a local frame of $\mathfrak{L}$ on a formal open subset $\mathfrak{U}\subseteq\mathfrak{X}$ we define $\|fs(\cdot)\|_{\mathfrak{L}}=|f(\cdot)|$ on $\mathfrak{U}^{\textup{an}}$ for any $f\in\Gamma(\mathfrak{U}^{\textup{an}},\mathcal{O}_{\mathfrak{X}}^{\textup{an}})$. As this is independent of the choice of $s$ and $\mathfrak{X}^{\textup{an}}$ is covered by such sets, this gives a well-defined metric on $L$. 
	\end{Def}
	\begin{Remark}
		\label{paracompact} We will work with paracompact (i.e. Hausdorff and every open cover has a locally finite refinement) strictly $K$-analytic spaces. As discussed in \cite[2.2]{GM} the category of these spaces is equivalent to the category of quasiseparated rigid analytic varieties over $K$ with a strictly $K$-affinoid G-covering of finite type (\cite[1.6]{Be4}). This allows us to apply Raynaud's theorem (\cite[Theorem 8.4.3]{Bo}) which shows that formal $K^\circ$-models of paracompact strictly $K$-analytic spaces exist and that the set of isomorphism classes of formal $K^\circ$-models is directed. 
	\end{Remark}
	\begin{Proposition}
		Let $X$ be a paracompact strictly $K$-analytic space, $L$ a line bundle on $X$ and $W$ a compact strictly $K$-analytic domain of $X$. Then every formal metric on $L\Big|_W$ extends to a formal metric on $L$.
	\end{Proposition}
	\begin{proof}
		\cite[Proposition 2.7]{GM}.
	\end{proof}
	\begin{Def}
		Let $X$ be a proper scheme over $K$ and $L$ a line bundle on $X$. An algebraic $K^\circ$-model of $X$ is a proper flat scheme $\mathscr{X}$ over $K^\circ$ with a fixed isomorphism from the generic fibre $\mathscr{X}_\eta$ to $X$. An algebraic $K^\circ$-model of $(X,L)$ is a pair $(\mathscr{X},\mathscr{L})$ where $\mathscr{X}$ is an algebraic $K^\circ$-model of $X$ and $\mathscr{L}$ is a line bundle on $\mathscr{X}$ with a fixed isomorphism from $\mathscr{L}\Big|_X$ to $L$. An algebraic $K^\circ$-model of $(X,L)$ gives rise to a formal $K^\circ$-model of $(X^{\textup{an}},L^{\textup{an}})$ by formal completion. Hence by the above, an algebraic model of $(X,L)$ induces a formal metric on $L^{\textup{an}}$. We call such metrics \textit{algebraic metrics}.
	\end{Def}
	\begin{Proposition}
		Let $X$ be a proper scheme over $K$ and $L$ a line bundle on $X$. Then a formal metric on $L^{\textup{an}}$ is the same as an algebraic metric.
	\end{Proposition}
	\begin{proof}
		\cite[Proposition 8.13]{GK2}, see also \cite[Remark 2.6]{GM}.
	\end{proof}
	\begin{Def}
		Let $X$ be a strictly $K$-analytic space and $L$ a line bundle on $X$. A metric $\|\cdot\|$ on $L$ is called \textit{piecewise linear} if there is a G-covering $(V_i)_{i\in I}$ and frames $s_i$ of $L$ over $V_i$ for every $i\in I$ such that $\|s_i(\cdot)\|=1$ on $V_i$.
	\end{Def}
	\begin{Proposition}
		Let $X$ be a strictly $K$-analytic space and $L$ a line bundle on $X$. Then
		\begin{enumerate}
			\item the isometry classes of piecewise linear metrics on line bundles on $X$ form an abelian group with respect to $\otimes$.
			\item the pull-back $f^\ast\|\cdot\|$ of a piecewise linear metric $\|\cdot\|$ on $L$ with respect to a morphism $f:Y\rightarrow X$ of strictly $K$-analytic spaces is a piecewise linear metric on $f^{\ast}L$.
			\item the minimum and the maximum of two piecewise linear metrics on $L$ are again piecewise linear metrics on $L$.
		\end{enumerate}
		\begin{proof}
			\cite[Proposition 2.12]{GM} (the proof does not use paracompactness).
		\end{proof}
	\end{Proposition}
	\begin{Proposition}
		Let $X$ be a paracompact strictly $K$-analytic space and $L$ a line bundle on $X$. Then a piecewise linear metric on $L$ is the same as a formal metric.
	\end{Proposition}
	\begin{proof}
		\cite[Proposition 2.10]{GM}.
	\end{proof}
	\begin{Def}
		Let $X$ be a strictly $K$-analytic space and $L$ a line bundle on $X$. A piecewise linear metric on $L$ is called \textit{semipositive} in $x\in X$ if there exists a compact strictly $K$-analytic domain $W$ which is a neighbourhood of $x$ such that there is a formal model $(\mathfrak{W},\mathfrak{L})$ of $\left(W,L\Big|_W\right)$ inducing the metric on $W$ and satisfying $\Deg_{\mathfrak{L}}(C)\geq0$ for every proper closed curve $C$ in the special fibre of $\mathfrak{W}$. The metric on $L$ is called semipositive in a subset $V\subseteq X$ if it is semipositive in every $x\in V$. It is called semipositive if it is semipositive in $X$.
	\end{Def}
	\begin{Proposition}
		Let $X$ be a paracompact strictly $K$-analytic space and $L$ a line bundle on $X$. A formal metric $\|\cdot\|$ on $L$ is semipositive in every $x\in X$ if and only if there exists a nef formal $K^\circ$-model $\mathfrak{L}$ of $L$ inducing $\|\cdot\|$. In particular we regain the original global definition of semipositivity by Zhang (\cite{Z}).
	\end{Proposition}
	\begin{proof}
		This is proved in \cite[Proposition 3.11]{GM} under the additional assumption that $X$ is separable, which was necessary in order to be able to use \cite[Lemme 6.5.1]{ChD}. Replacing this with Corollary \ref{CDLemma}, the same proof applies to the more general case.
	\end{proof}
	\begin{Proposition}
		\label{GM Proposition 3.12} Let $X$ be a proper scheme over $K$ and $L$ a line bundle on $X$. Let $\|\cdot\|_1,\|\cdot\|_2$ be two piecewise linear metrics on $L^{\textup{an}}$ which are semipositive in $x\in X^{\textup{an}}$. Then $\|\cdot\|:=\min(\|\cdot\|_1,\|\cdot\|_2)$ is semipositive in $x$.
	\end{Proposition}
	\begin{proof}
		\cite[Proposition 3.12]{GM}.
	\end{proof}
	\begin{Def}
		Let $X$ be a strictly $K$-analytic space and $L$ a line bundle on $X$. A metric $\|\cdot\|$ on $L$ is called \textit{piecewise }$\mathbb{Q}$\textit{-linear} if for every $x\in X$ there is an open neighbourhood $W$ of $x$ and a non-zero $n\in\mathbb{N}$ such that $\|\cdot\|^{\otimes n}\Big|_W$ is a piecewise linear metric on $L^{\otimes n}\Big|_W$. \\
		A piecewise $\mathbb{Q}$-linear metric on $L$ is called \textit{semipositive} in $x\in X$ if in the above $\|\cdot\|^{\otimes n}\Big|_W$ is semipositive in $x$.
	\end{Def}
	\begin{Proposition}
		\label{dense} Let $X$ be a paracompact strictly $K$-analytic space and $L$ a line bundle on $X$. Any continuous metric on $L$ can be uniformly approximated by piecewise $\mathbb{Q}$-linear metrics on $L$.
	\end{Proposition}	
	\begin{proof}
		\cite[Theorem 2.17]{GM}.
	\end{proof}
	\section{Measures}
	\label{measures} We recall the real Monge-Ampère operator which associates to a convex function a positive Borel measure. Then we introduce the Chambert-Loir measure on the generic fibres of admissible formal schemes and on paracompact strictly $K$-analytic spaces. Chambert-Loir introduced these measures in \cite{Ch} on the analytification $X^{\textup{an}}$ of a proper variety $X$ over $K$ under the assumption that $K$ has a countable dense subfield and associates to a family of semipositive metrized line bundles a positive Radon measure. This was later extended by Gubler to the case of an algebraically closed base field in \cite{Gub1}. Using the local approach to metrics from section \ref{Metrics}, it is now possible to define Monge-Ampère measures locally. Note that there is also a local approach by Chambert-Loir and Ducros in \cite{ChD} which associates a measure to a metric which is locally psh-approximable. However it is not known whether a semipositive metric is locally psh-approximable. In this section we assume that the non-archimedean complete base field $K$ is algebraically closed which is no restriction as one can always reduce to this case by base change (see Remark \ref{extend}).
	\begin{Def}
		Let $\Omega\subseteq\mathbb{R}^n$ be bounded, open and convex and denote by $\lambda$ the standard Lebesgue measure on $\mathbb{R}^n$ and by $\langle\cdot,\cdot\rangle$ the standard scalar product on $\mathbb{R}^n$. Let $h$ be a convex function on $\Omega$ and $x_0\in\Omega$. We define the \textit{gradient image} of $x_0$ under $h$ to be
		\[
			\nabla h(x_0):=\left\{p\in\mathbb{R}^n\;\Big|\;\forall x\in\Omega\;:\;h(x_0)+\langle x-x_0,p\rangle\leq h(x)\right\}
		\]
		and for $E\subseteq\Omega$
		\[
			\nabla h(E):=\bigcup_{x_0\in E}\nabla h(x_0).
		\]
		Note that if $E$ is a Borel set, the same is true for $\nabla h(E)$. Finally we define the Monge-Ampère measure associated to $h$ by 
		\[
			\MA(h)(E):=\lambda(\nabla h(E))
		\]
		for all Borel sets $E\subseteq\Omega$. It is indeed a measure on the Borel $\sigma$-algebra, for details see \cite[Section 2]{RT}. The real Monge-Ampère operator is continuous in the sense that if $(u_n)_{n\in\mathbb{N}}$ is a sequence of convex functions on $\Omega$ converging pointwise to a convex function $u$ then $(\MA(u_n))_{n\in\mathbb{N}}$ converges weakly to $\MA(u)$. If $h$ is two times continuously differentiable then $\MA(h)=\det D^2 h \cdot\lambda$.
		\end{Def}
		\begin{Def}
			\label{irreduciblecomponents} In \cite[Definition 2.2.2]{Co} Conrad defined the notion of irreducibility for analytic spaces which we recall here. Let $X$ be a paracompact strictly $K$-analytic space and $p:\tilde{X}\rightarrow X$ the normalization of $X$ (\cite[2.1]{Co}). Then the irreducible components of $X$ are defined to be the sets $X_i:=p(\tilde{X}_i)$ where $\tilde{X}_i$ are the connected components of $\tilde{X}$. The space $X$ is said to be irreducible if it has a unique irreducible component. By \cite[Lemma 2.2.3]{Co} $X$ is irreducible if and only if it can not non trivially be written as a union of two closed strictly $K$-analytic subsets. \\
			Let $Y$ be an irreducible component of $X$ and $V=\mathscr{M}(\mathscr{A})$ an affinoid domain with $Y\cap V\neq\emptyset$. Then by \cite[Corollary 2.2.9]{Co} there is an irreducible component $Y'$ of $V$ which is contained in $V\cap Y$. Then $Y'$ corresponds to a minimal prime ideal $\mathfrak{p}$ of $\mathscr{A}$ and hence to an irreducible component of $\Spec(\mathscr{A})$. We define the \textit{multiplicity of $Y$} to be the multiplicity of this component. Note that this does not depend on the choice of $V$ and $Y'$: If $V'=\mathscr{M}(\mathscr{B})\subseteq V$ and $\mathfrak{p}'$ is a minimal prime ideal of $\mathscr{B}$ lying over $\mathfrak{p}$ then $\mathscr{B}/\mathfrak{p}\mathscr{B}$ is reduced by \cite[Corollary 7.3.2/10]{BGR} as it induces an affinoid domain in $\mathscr{M}(\mathscr{A}/\mathfrak{p})$ which is reduced. Hence also $\mathscr{B}_{\mathfrak{p}'}/\mathfrak{p}\mathscr{B}_{\mathfrak{p}'}$ is reduced and since $\mathscr{B}_{\mathfrak{p}'}$ is a local ring of dimension 0, this implies $\mathfrak{p}'\mathscr{B}_{\mathfrak{p}'}=\mathfrak{p}\mathscr{B}_{\mathfrak{p}'}$. Hence by \cite[Lemma A.4.1]{Fu2} the multiplicity of the irreducible component corresponding to $\mathfrak{p}$ is equal to that of the irreducible component corresponding to $\mathfrak{p}'$. \\
			Let $\varphi:X\rightarrow Y$ be a proper surjective morphism of irreducible and reduced strictly $K$-analytic spaces. If $\dim(Y)<\dim(X)$ we set $\deg(\varphi)=0$. Otherwise $\varphi$ is a finite morphism outside a lower dimensional analytic subset $W$ of $Y$. Let $\mathscr{M}(\mathscr{A}')$ be an affinoid domain in $Y\setminus W$, $V$ an irreducible component of $\Spec(\mathscr{A}')$ and $\mathscr{M}(\mathscr{A}):=\varphi^{-1}(\mathscr{M}(\mathscr{A}'))$ then $\Spec(\mathscr{A})\rightarrow\Spec(\mathscr{A}')$ is finite and we define $\deg(\varphi)$ to be the sum of the degrees of the irreducible components of $\Spec(\mathscr{A})$ over $V$. As explained in \cite[2.6]{Gub4} this again does not depend on the choices.
		\end{Def}
		\refstepcounter{Def}
	\textit{\theDef\;Monge-Ampère measure for line bundles on admissible formal schemes} \\ 
		 \label{measure} Let $\mathfrak{X}$ be an admissible formal scheme over $K^\circ$ of dimension $n+1$ with generic fibre $X$. Our goal is to introduce a Monge-Ampère measure on $X$ for formal line bundles $\mathfrak{L}_1,...,\mathfrak{L}_n$ on $\mathfrak{X}$. We assume first that $X$ is irreducible and reduced and that the special fibre of $\mathfrak{X}$ is reduced. Then the non-archimedean Monge-Ampère measure on $X$ with respect to these metrized line bundles is defined as
		\[
			c_1(\mathfrak{L}_1)\wedge...\wedge c_1(\mathfrak{L}_n):=\sum_{\substack{Y\in\irr(\tilde{\mathfrak{X}}) \\ Y\text{ proper}}} \Deg_{\mathfrak{L}_1,...,\mathfrak{L}_n}(Y)\cdot \delta_{\zeta_Y},	
		\]
		where $\delta_{\zeta_Y}$ denotes the Dirac-measure at the unique point $\zeta_Y$ which is mapped to the generic point of the proper irreducible component $Y$ under the reduction map (cf. \cite[Proposition 2.4.4]{Be1}).
		
		If $\mathfrak{X}$ has irreducible and reduced generic fibre but no longer reduced special fibre, there is a canonical admissible formal model $\mathfrak{X}'$ of $X$ with reduced special fibre together with a finite morphism $\iota:\mathfrak{X}'\rightarrow\mathfrak{X}$ which restricts to the identity on $X$ which can be constructed as follows (cf. \cite[Definition 3.10]{Gub4}). Choose a cover $(\mathfrak{U}_i=\Spf(A_i))_{i\in I}$ of $\mathfrak{X}$ by affine formal subschemes. Define $\mathscr{A}_i:=A\otimes_{K^\circ} K$. If $\Spf(B)\subseteq\Spf(A_i)$ is a formal open subscheme for some $i\in I$ then $A_i\rightarrow B$ induces a morphism $\mathscr{A}_i^\circ\rightarrow\mathscr{B}^\circ$ for $\mathscr{B}:=B\otimes_{K^\circ}K$. Hence by standard arguments we can glue the $\Spf(\mathscr{A}_i^\circ)$ to obtain $\mathfrak{X}'$ and the canonical morphisms $A_i\rightarrow\mathscr{A}_i^\circ$ induce the morphism $\mathfrak{X}'\rightarrow\mathfrak{X}$. We then define
		\[
			c_1(\mathfrak{L}_1)\wedge...\wedge c_1(\mathfrak{L}_n):=(\iota^{\textup{an}})_{\ast}(c_1(\iota^{\ast}\mathfrak{L}_1)\wedge...\wedge c_1(\iota^{\ast}\mathfrak{L}_n)).
		\]		
		In the general case, let $X=\sum_j m_jX_j$ be the decomposition of the generic fibre into prime cycles. By \cite[Proposition 3.3]{Gub4} the closure $\overline{X}_j$ of $X_j$ in $\mathfrak{X}$ is an admissible formal scheme with irreducible and reduced generic fibre $X_j$. We define
		\[
			c_1(\mathfrak{L}_1)\wedge...\wedge c_1(\mathfrak{L}_n):=\sum_j m_j\cdot c_1\left(\mathfrak{L}_1\Big|_{\overline{X}_j}\right)\wedge...\wedge c_1\left(\mathfrak{L}_n\Big|_{\overline{X}_j}\right)
		\]
		as a measure on $X$.
	\begin{Remark}
		There is a close connection of the Monge-Ampère measure with the intersection product on formal schemes as defined in \cite{Gub4}: Assume that $\mathfrak{X}$ has irreducible, reduced and boundaryless generic fibre and reduced special fibre. In addition to $\mathfrak{L}_1,...,\mathfrak{L}_n$ let $\mathfrak{L}_0$ be a formal line bundle on $\mathfrak{X}$ which is trivial on the generic fibre and set $f:=-\log\|1\|$ where $\|\cdot\|$ is the formal metric induced by $\mathfrak{L}_0$. Suppose that $f$ has compact support and let $D:=\Div(1)$ be the Cartier divisor on $\mathfrak{X}$ induced by $1$ as in \cite[Remark 3.1]{Gub4}. We examine the Weil divisor $\cyc(D)$ associated to $D$ as defined in \cite[§3]{Gub4}. Since $\mathfrak{L}_0$ is trivial on the generic fibre, the horizontal part of $\cyc(D)$ is zero while the vertical part is by definition (\cite[3.8]{Gub4}) given by $\sum_{Y\in\irr(\tilde{\mathfrak{X}})}f(\zeta_Y)\cdot Y$. Now since $\mathfrak{X}^{\textup{an}}$ has no boundary, every irreducible component of $\tilde{\mathfrak{X}}$ is proper by Corollary \ref{CDLemma} and together with the definition of the intersection product (\cite[§4]{Gub4}) we obtain
		\[
			\int_{\mathfrak{X}^{\textup{an}}}f c_1(\mathfrak{L}_1)\wedge...\wedge c_1(\mathfrak{L}_n)=\sum_{Y\in\irr(\tilde{\mathfrak{X}})}f(\zeta_Y)\cdot\Deg_{\mathfrak{L}_1,...,\mathfrak{L}_n}(Y)=\Deg_{\mathfrak{L}_1,...,\mathfrak{L}_n}(\cyc(D)).
		\]	
	\end{Remark}
	\begin{Proposition}
		\label{firstproperties} The measure defined above has the following properties:
		\begin{enumerate}
			\item $c_1(\mathfrak{L}_1)\wedge...\wedge c_1(\mathfrak{L}_n)$ is a discrete measure (i.e. of the form $\sum_{x\in S} \lambda_x\delta_x$ with $S\subseteq X$ a closed discrete subset, $\lambda_x\in\mathbb{R}$ and $\delta_x$ the Dirac-measure at $x$) whose support is contained in the relative interior of $X$ over $K$ (in the sense of \cite[1.5]{Be4}).
			\item $c_1(\mathfrak{L}_1)\wedge...\wedge c_1(\mathfrak{L}_n)$ is multilinear and symmetric in $\mathfrak{L}_1,...,\mathfrak{L}_n$.
			\item Let $\varphi:\mathfrak{X}'\rightarrow \mathfrak{X}$ be a proper morphism of admissible formal schemes over $K^\circ$ with irreducible and reduced generic fibres of dimension $n$ such that the induced morphism on the generic fibres is surjective. Then for formal line bundles $\mathfrak{L}_1,...,\mathfrak{L}_n$ on $\mathfrak{X}$ we have
			\[
				(\varphi^{\textup{an}})_\ast \left(c_1(\varphi^{\ast}\mathfrak{L}_1)\wedge...\wedge c_1(\varphi^{\ast}\mathfrak{L}_n)\right)=\Deg(\varphi^{\textup{an}})c_1(\mathfrak{L}_1)\wedge...\wedge c_1(\mathfrak{L}_n).
			\]
		\end{enumerate}
	\end{Proposition}
	\begin{proof}
		 ii) follows from symmetry and multilinearity of the intersection product (\cite[Proposition 2.5]{Fu2}). For iii) we reduce first to the case where $\mathfrak{X}'$ and $\mathfrak{X}$ have reduced special fibre. Let $\mathfrak{Y}'$ respectively $\mathfrak{Y}$ be the canonical formal models with reduced special fibre as in \ref{measure}. This construction is functorial and we obtain a commutative diagram
		 \[
		 	\xymatrix{
				\mathfrak{Y}' \ar[d]_{\iota'} \ar[r]^{\varphi'} & \mathfrak{Y} \ar[d]_{\iota} \\
				\mathfrak{X}' \ar[r]_{\varphi} & \mathfrak{X}
			}
		 \]
		 Assuming that we know the claim for reduced special fibres we obtain
		 \begin{align*}
		 	\Deg(\varphi^{\textup{an}})c_1(\mathfrak{L}_1)\wedge...\wedge c_1(\mathfrak{L}_n)&=\Deg(\varphi^{\prime an}) (\iota^{\textup{an}})_\ast (c_1(\iota^{\ast}\mathfrak{L}_1)\wedge...\wedge c_1(\iota^{\ast}\mathfrak{L}_n)) \\
		 	&=(\iota^{\textup{an}})_{\ast}(\varphi^{\prime an})_{\ast}(c_1(\varphi^{\prime \ast}\iota^{\ast}\mathfrak{L}_1)\wedge...\wedge c_1(\varphi^{\prime \ast}\iota^{\ast}\mathfrak{L}_n)) \\
		 	&=(\varphi^{\textup{an}})_{\ast}(\iota^{\prime an})_{\ast}(c_1(\iota^{\prime \ast}\varphi^{\ast}\mathfrak{L}_1)\wedge...\wedge c_1(\iota^{\prime \ast}\varphi^{\ast}\mathfrak{L}_n)) \\
		 	&=(\varphi^{\textup{an}})_{\ast} (c_1(\varphi^{\ast}\mathfrak{L}_1)\wedge...\wedge c_1(\varphi^{\ast}\mathfrak{L}_n)).
		 \end{align*} 
		 So from now on assume that $\mathfrak{X}'$ and $\mathfrak{X}$ have reduced special fibre. Let $Y$ be an irreducible component of $\tilde{\mathfrak{X}}$ with corresponding Shilov point $\zeta_Y$. Let $\zeta_1,...,\zeta_r$ be the preimages of $\zeta_Y$ under $\varphi^{\textup{an}}$ with corresponding irreducible components $Y_1,...,Y_r$ of $\tilde{\mathfrak{X}'}$. If $Y$ is proper then clearly all the $Y_i$ are proper. If on the other hand one of the $Y_i$ is proper then $Y$ is proper by \cite[Proposition 12.59]{GW}. In this case we can use the projection formula to calculate:
		 \begin{align*}
		 		(\varphi^{\textup{an}})_{\ast}\left(c_1(\varphi^{\ast}\mathfrak{L}_1)\wedge...\wedge c_1(\varphi^{\ast}\mathfrak{L}_n)\right)(\zeta_Y)&=\sum_{i=1}^r c_1(\varphi^{\ast}\mathfrak{L}_1)\wedge...\wedge c_1(\varphi^{\ast}\mathfrak{L}_n)(\zeta_i) \\
		 		&=\sum_{i=1}^r \Deg_{\mathfrak{L}_1,...,\mathfrak{L}_n}(\tilde{\varphi}_{\ast}Y_i) \\
		 		&=\sum_{i=1}^r \Deg_{\mathfrak{L}_1,...,\mathfrak{L}_n}(Y)\cdot [\tilde{K}(Y_i):\tilde{K}(Y)]
		 \end{align*}
		 As already mentioned in Definition \ref{irreduciblecomponents}, $\varphi^{\textup{an}}$ is finite outside a lower dimensional analytic subset. Hence we may apply equation (3) in the proof of \cite[Proposition 4.5]{Gub4} to see that the last term in the display equals $\Deg(\varphi^{\textup{an}})\cdot c_1(\mathfrak{L}_1)\wedge...\wedge c_1(\mathfrak{L}_n)(\zeta_Y)$. \\
		 On the other hand, if $Y$ is an irreducible component of $\tilde{\mathfrak{X}}'$ whose image is not an irreducible component of $\tilde{\mathfrak{X}}$ then its degree with respect to the line bundles $\varphi^{\ast}\mathfrak{L}_1,...,\varphi^{\ast}\mathfrak{L}_n$ is $0$ by the projection formula, as the image is of lower dimension. This proves iii). \\ 
		 For i) let $\mathfrak{X}^{\textup{an}}=\sum_j m_j X_j$ be the decomposition into prime cycles. It is then enough to prove the claim for each $X_j$ and by definition of the measure we may hence assume that $\mathfrak{X}$ has irreducible and reduced generic fibre and reduced special fibre. Let $S$ be the set of all $\zeta_Y$ where $Y$ is a proper irreducible component of $\tilde{\mathfrak{X}}$ with $\Deg_{\mathfrak{L}_1,...,\mathfrak{L}_n}(Y)\neq 0$. Then $S$ is discrete as $\red^{-1}(Y)$ is an open neighbourhood of $\zeta_Y$ which does not contain any other points of $S$. Furthermore $X$ is the union of all $\red^{-1}(Y)$ where $Y$ runs over all irreducible components of $\tilde{\mathfrak{X}}$ and as all of these sets contain at most one point of $S$ and by paracompactness of $X$, every $x\notin S$ has an open neighbourhood which does not intersect $S$ and hence $S$ is closed. By definition $c_1(\mathfrak{L}_1)\wedge...\wedge c_1(\mathfrak{L}_n)$ is of the desired form and its support is contained in the relative interior of $X$ over $K$ by Corollary \ref{CDLemma}.
	\end{proof}
	\begin{Lemma}
		\label{switch} Let $\mathfrak{X}$ be an admissible formal scheme over $K^\circ$ of dimension $n+1$ with boundaryless generic fibre $\mathfrak{X}^{\textup{an}}$ and $L_0,...,L_n$ line bundles on $\mathfrak{X}^{\textup{an}}$ endowed with formal metrics corresponding to the models $\mathfrak{L}_0,...,\mathfrak{L}_n$ on $\mathfrak{X}$. Suppose that $L_0=L_1=\mathcal{O}_{\mathfrak{X}^{\textup{an}}}$, denote by $\|\cdot\|_0$ and $\|\cdot\|_1$ the metrics on $L_0$ respectively $L_1$ and set $f_0:=-\log\|1\|_0$, $f_1:=-\log\|1\|_1$. Suppose that $f_0$ and $f_1$ have compact support. Then
		\[
			\int_{\mathfrak{X}^{\textup{an}}} f_0 \;c_1(\mathfrak{L_1})\wedge...\wedge c_1(\mathfrak{L}_n)=\int_{\mathfrak{X}^{\textup{an}}} f_1 \;c_1(\mathfrak{L}_0)\wedge c_1(\mathfrak{L}_2)\wedge...\wedge c_1(\mathfrak{L}_n).
		\]
	\end{Lemma}
	\begin{proof}
		%This follows from the symmetry of local heights, see \cite[Theorem 3.5 (a)]{Gub1}.
		Let $\mathfrak{X}^{\textup{an}}=\sum_j m_j X_j$ be the decomposition into prime cycles. It is enough to prove the claim for the closures $\overline{X}_j$ of $X_j$ in $\mathfrak{X}$. We may hence assume that $\mathfrak{X}^{\textup{an}}$ is irreducible and reduced. Furthermore by passing to a dominating model as in \ref{measure}, we may assume that the special fibre $\tilde{\mathfrak{X}}$ of $\mathfrak{X}$ is reduced. As $\mathfrak{X}^{\textup{an}}$ has no boundary, every irreducible component of $\tilde{\mathfrak{X}}$ is proper by Corollary \ref{CDLemma} and hence using commutativity of the intersection product (\cite[Theorem 5.9]{Gub4}) we obtain
		\begin{align*}
			\int_{\mathfrak{X}^{\textup{an}}} f_0\;c_1(\mathfrak{L}_1)\wedge ... \wedge c_1(\mathfrak{L}_n)&=\sum_{\substack{Y\in\irr(\tilde{\mathfrak{X}})}} f_0(\zeta_Y)\cdot \Deg_{\mathfrak{L}_1,...,\mathfrak{L}_n}(Y) \\
			&= \Deg_{\mathfrak{L}_1,...,\mathfrak{L}_n}\left(\cyc(\Div_{\mathfrak{L}_0}(1))\right) \\
			&= \Deg_{\mathfrak{L}_0, \mathfrak{L}_2,...,\mathfrak{L}_n}\left(\cyc(\Div_{\mathfrak{L}_1}(1))\right) \\
			&=\sum_{\substack{Y\in\irr(\tilde{\mathfrak{X}})}} f_1(\zeta_Y)\cdot \Deg_{\mathfrak{L}_0,\mathfrak{L}_2,...,\mathfrak{L}_n}(Y) \\
			&=\int_{\mathfrak{X}^{\textup{an}}} f_1\;c_1(\mathfrak{L}_0)\wedge c_1(\mathfrak{L}_2)\wedge...\wedge c_1(\mathfrak{L}_n).
		\end{align*}
	\end{proof}
	\begin{Def}
		Let $X$ be an $n$-dimensional paracompact strictly $K$-analytic space and $\overline{L}_1,...,\overline{L}_n$ formally metrized line bundles on $X$. Let $\mathfrak{X}$ be a formal model of $X$ on which there exist formal models $\mathfrak{L}_1,...,\mathfrak{L}_n$ of $\overline{L}_1,...,\overline{L}_n$. The existence of such a formal model follows from Remark \ref{paracompact}. We then define
		\[
			c_1(\overline{L}_1)\wedge...\wedge c_1(\overline{L}_n):=c_1(\mathfrak{L}_1)\wedge...\wedge c_1(\mathfrak{L}_n).
		\]
		Note that this definition is independent of the choice of $\mathfrak{X}$ and $\mathfrak{L}_1,...,\mathfrak{L}_n$ by the projection formula. If the metrics on $\overline{L}_1,...,\overline{L}_n$ are semipositive then $c_1(\overline{L}_1)\wedge...\wedge c_1(\overline{L}_n)$ is a positive measure.
	\end{Def}
	\begin{Lemma}
		\label{local} Let $W_2$ be a paracompact strictly $K$-analytic space of dimension $n$ and $W_1\subseteq W_2$ a paracompact strictly $K$-analytic subdomain of $W_2$. Then for formally metrized line bundles $\overline{L}_1,...,\overline{L}_n$ on $W_2$, we have $c_1(\overline{L}_1)\wedge...\wedge c_1(\overline{L}_n)=c_1\left(\overline{L}_1\Big|_{W_1}\right)\wedge...\wedge c_1\left(\overline{L}_n\Big|_{W_1}\right)$ in the topological interior $\overset{\circ}{W_1}$ of $W_1$ in $W_2$.
	\end{Lemma}
	\begin{proof}
		Let $W_2=\sum_j m_j X_j$ be the decomposition of $W_2$ into prime cycles and for each $j$ let $(X_{ij})_{i\in I_j}$ be the irreducible components of $W_1$ with $X_{ij}\subseteq X_j\cap W_1$. Then $W_1=\sum_{j,i} m_jX_{ij}$ is the decomposition of $W_1$ into prime cycles. Furthermore, the intersection of any two irreducible components of $W_1$ does not contain a Shilov point as it is of lower dimension and hence does not meet the support of the measures of interest. By linearity in the irreducible components we may therefore assume that $W_1$ and $W_2$ are irreducible and reduced. Let $\mathfrak{X}_2$ be a formal model of $W_2$ with reduced special fibre on which there exist formal models of $\overline{L}_1,...,\overline{L}_n$. Let $\mathfrak{X}_1$ be a formal model of $W_1$ which exists by paracompactness of $W_1$, see Remark \ref{paracompact}. After possibly blowing up, the inclusion $W_1\hookrightarrow W_2$ induces a morphism $\iota:\mathfrak{X}_1\rightarrow\mathfrak{X}_2$ (\cite[Theorem 8.4.3]{Bo}). Let $x\in\overset{\circ}{W_1}$. As both measures are discrete it is enough to show that they have the same mass at $x$. Let $\Int(W_i)$ denote the relative interior of $W_i$ over $K$ in the sense of \cite[1.5]{Be4}. If $x\in\Int(W_2)$ then $x\in\Int(W_1)$ by \cite[Proposition 1.5.5 (ii)]{Be4}. Conversely if $x\in\Int(W_1)$ then there exists an affinoid neighbourhood $V$ of $x$ in $W_1$ such that $x$ is in the relative interior of $V$ over $K$. But $V$ is also a neighbourhood of $x$ in $W_2$ as $x\in\overset{\circ}{W_1}$ and therefore $x\in\Int(W_2)$. Hence $x\in\Int(W_1)$ if and only if $x\in\Int(W_2)$. If this is not the case then by definition of the measures and Corollary \ref{CDLemma}, both of them are zero at $x$. So assume that $x\in\Int(W_1)$. Choose a locally finite cover $(\mathfrak{U}_i)_{i\in I}$ of $\mathfrak{X}_1$ by open affine formal subschemes and let $\mathfrak{U}$ be the union of all $\mathfrak{U}_i$ which contain $\red(x)$. Then $\mathfrak{U}$ is an open and quasi-compact formal subscheme of $\mathfrak{X}_1$. Analogously choose a cover $(\mathfrak{V}_i)_{i\in J}$ of $\mathfrak{X}_2$ by open affine formal subschemes. As $\iota(\mathfrak{U})$ is quasi-compact, there is a finite subcover of it. Let $\mathfrak{V}$ be the union of the sets in this subcover and add all $\mathfrak{V}_i$ with $\red(x)\in \mathfrak{V}_i$. Then also $\mathfrak{V}$ is an open and quasi-compact formal subscheme of $\mathfrak{X}_2$ and $\iota$ induces a morphism $\mathfrak{U}\rightarrow\mathfrak{V}$. By \cite[Corollary 5.4]{BL} there is an admissible formal blowing up $\mathfrak{V}'\rightarrow\mathfrak{V}$ such that the induced morphism $\mathfrak{U}'\rightarrow\mathfrak{V}'$ is an open immersion. \\
		Let $Y$ be an irreducible component of $\tilde{\mathfrak{X}}_2$ with corresponding divisorial point $\zeta_{Y}=x$. Then $Y\subseteq\tilde{\mathfrak{V}}$ by definition and hence we may calculate the mass of $c_1(\overline{L}_1)\wedge...\wedge c_1(\overline{L}_n)$ at $x$ using $\mathfrak{V}$. By Proposition \ref{firstproperties} iii) we may also use $\mathfrak{V}'$. So let $Y'$ be the irreducible component of $\tilde{\mathfrak{V}}'$ corresponding to $x$. Since $\red(x)\in\tilde{\mathfrak{U}}'$ we see that $Y\cap\tilde{\mathfrak{U}}'$ is an irreducible component of $\tilde{\mathfrak{U}}'$. Additionally, by Corollary \ref{CDLemma}, $Y'$ and $Y'\cap\tilde{\mathfrak{U}}'$ are proper and hence $Y'=Y'\cap\tilde{\mathfrak{U}}'$ and it is an irreducible component of $\tilde{\mathfrak{U}}'$. It's image in $\tilde{\mathfrak{U}}$ is a proper irreducible component of $\tilde{\mathfrak{U}}$ and hence also an irreducible component of $\tilde{\mathfrak{X}}_1$. By the same argumentation as above we may use $\mathfrak{U}'$ instead of $\mathfrak{X}_1$ to calculate the mass of $c_1\left(\overline{L}_1\Big|_{W_1}\right)\wedge...\wedge c_1\left(\overline{L}_n\Big|_{W_1}\right)$ at $x$. This shows that the mass of the two measures is equal at $x$ in this case. \\
		Conversely, if $Y$ is an irreducible component of $\tilde{\mathfrak{X}}_1$ with corresponding divisorial point $\zeta_Y=x$ then $Y\subseteq\tilde{\mathfrak{U}}$ by definition. Again we may use $\mathfrak{U}'$ to calculate the mass at $x$ and we denote the corresponding irreducible component by $Y'$. Then the closure $\overline{Y}'$ of $Y'$ in $\tilde{\mathfrak{V}}'$ is an irreducible component of $\tilde{\mathfrak{V}}'$ with corresponding divisorial point $\zeta_{\overline{Y}'}=x$ and hence by the above $\overline{Y}'=Y'$. Therefore $c_1(\overline{L}_1)\wedge...\wedge c_1(\overline{L}_n)$ and $c_1\left(\overline{L}_1\Big|_{W_1}\right)\wedge...\wedge c_1\left(\overline{L}_n\Big|_{W_1}\right)$ coincide at $x$.
	\end{proof}
	\begin{Def}
		Let $X$ be a Hausdorff topological space. A measure $\mu$ on the $\sigma$-algebra of Borel sets of $X$ is called a \textit{Radon measure} if
		\begin{enumerate}
			\item for every $x\in X$ there exists an open neighbourhood $U$ of $X$ with $\mu(U)<\infty$,
			\item for every open set $U\subseteq X$ we have $\mu(U)=\sup\left\{\mu(K)\;\Big|\;K\subseteq U,\;K\text{ compact}\right\}$,
			\item for every Borel set $B$ of $X$ we have $\mu(B)=\inf\left\{\mu(U)\;\Big|\;B\subseteq U,\;U\text{ open}\right\}$.
		\end{enumerate}
	\end{Def}
	\begin{Remark}
		It follows from Proposition \ref{firstproperties} i) that the measure defined in \ref{measure} is a Radon measure.
	\end{Remark}
	\begin{Def}
		\label{measureopen} Let $V$ be a strictly $K$-analytic Hausdorff space of dimension $n$ and $\overline{L}_1,...,\overline{L}_n$ semipositive piecewise $\mathbb{Q}$-linear metrized line bundles on $V$. The assignment
		\begin{align*}
			C_c(V)&\rightarrow\mathbb{R}_{\geq 0}, \\
			f&\mapsto\frac{1}{e_1\cdot...\cdot e_n}\int_{W} f\; c_1\left(\overline{L}_1^{e_1}\Big|_{W}\right)\wedge...\wedge c_1\left(\overline{L}_n^{e_n}\Big|_{W}\right)
		\end{align*}
		where $W$ is a compact strictly $K$-analytic domain with $\supp(f)\subseteq \overset{\circ}{W}$ and $e_1,...,e_n\in\mathbb{N}$ are non-zero integers such that $\overline{L}_i^{e_i}\Big|_{W}$ is a formally metrized line bundle, yields a positive linear functional on the space $C_c(V)$ of continuous functions with compact support in $V$ and hence by the Riesz Representation Theorem (see \cite[Theorem 2.14]{Ru}) a positive Radon measure on $V$ which we again denote by $c_1(\overline{L}_1)\wedge...\wedge c_1(\overline{L}_n)$. Note that the integral does neither depend on the choice of $W$ by Lemma \ref{local} nor on the choice of the $e_i$ by Proposition \ref{firstproperties} and that we can always find such a $W$ together with the $e_i$ by choosing for every point in $\supp(f)$ a compact strictly $K$-analytic neighbourhood where some powers of the $\overline{L}_i$ are formally metrized and using compactness of $\supp(f)$.
	\end{Def}
	\begin{Remark}
		It is easy to see that Proposition \ref{firstproperties}, Lemma \ref{switch} and Lemma \ref{local} remain true if we replace formal metrics by piecewise $\mathbb{Q}$-linear metrics.
	\end{Remark}
	\begin{Proposition}
		\label{convergence}
		Let $X$ be a separated scheme of finite type over $K$ of dimension $n$ with line bundles $L_1,...,L_n$ on $X$. Let $V$ be an open subset of $X^{\textup{an}}$ and $\|\cdot\|_i$ a continuous metric on $L_i^{\textup{an}}\Big|_{V}$ for each $i$. Denote by $\overline{L}_1,...,\overline{L}_n$ the line bundles $L_1^{\textup{an}}\Big|_V,...,L_n^{\textup{an}}\Big|_V$, endowed with these metrics. For $i\in\{1,...,n\}$ let $(\|\cdot\|_{i,k})_{k\in\mathbb{N}}$ be piecewise $\mathbb{Q}$-linear metrics on $L_i\Big|_V$ converging uniformly to the continuous metric $\|\cdot\|_i$ on $L_i\Big|_{V}$. Suppose that all $\|\cdot\|_{i,k}$ are semipositive in $V$. Denote by $\overline{L}_{i,k}$ the line bundle $L_i^{\textup{an}}\Big|_{V}$ endowed with the metric $\|\cdot\|_{i,k}$. Then the measures $c_1\left(\overline{L}_{1,k}\right)\wedge ...\wedge c_1\left(\overline{L}_{n,k}\right)$ converge weakly to a positive Radon measure on $V$.
	\end{Proposition}
	\begin{proof}
		By Vojta's version of Nagata's compactification theorem (\cite[Theorem 5.7]{V}) we may assume that $X$ is proper. We show by reverse induction over $m\in\{0,...,n\}$ that the claim holds when for some choice of pairwise different $i_1,...,i_n\in\{1,...,n\}$ the sequences $\Big(\|\cdot\|_{i_1,k}\Big)_{k\in\mathbb{N}},...,\Big(\|\cdot\|_{i_m,k}\Big)_{k\in\mathbb{N}}$ are constant with respect to $k$. The case $m=n$ is clear. So let $0\leq m < n$ and assume that the claim holds for $m+1$. For $j\in\{m+1,...,n\}$ we can write $\|\cdot\|_{i_j,k}=\|\cdot\|_{i_j,1}\otimes \|\cdot\|'_{j,k}$ for a sequence of piecewise $\mathbb{Q}$-linear metrics $\Big(\|\cdot\|'_{j,k}\Big)_{k\in\mathbb{N}}$ on $\mathcal{O}_{X^{\textup{an}}}\Big|_{V}$ converging uniformly to a continuous metric $\|\cdot\|'_j$ on $\mathcal{O}_{X^{\textup{an}}}\Big|_{V}$. Denote by $\overline{\mathcal{O}}_{j,k}$ the line bundle $\mathcal{O}_{X^{\textup{an}}}\Big|_{V}$ endowed with the metric $\|\cdot\|'_{j,k}$. We show that 
		\[
			\left(\mu_{m,k}:=c_1(\overline{L}_{i_1,1})\wedge...\wedge c_1(\overline{L}_{i_m,1})\wedge c_1(\overline{L}_{i_{m+1},k})\wedge...\wedge c_1(\overline{L}_{i_n,k})\right)_{k\in\mathbb{N}}
		\]
		is a Cauchy sequence with respect to the weak topology on the space of Borel-measures on $V$. Thus we have to show that for all continuous functions $f$ on $X$ with compact support in $V$:
		\[
			\left|\int_{V} f\;\mu_{m,k} - \int_{V} f\;\mu_{m,k'}\right| \underset{k,k'\rightarrow\infty}{\longrightarrow} 0.
		\]
		Let $W$ be a compact strictly $K$-analytic domain with $\supp(f)\subseteq\overset{\circ}{W}$ and $W\subseteq V$. By \cite[Proposition 2.7]{GM} we may extend the metrics from $W$ to $X^{\textup{an}}$ and hence assume that they are defined on the whole space. Hence by Chow's lemma and the projection formula we may assume that $X$ is projective. Then by \cite[Proposition 10.5]{Gub5} any formal model of $X$ is dominated by a projective model. Any formal line bundle on this model becomes semipositive after tensoring with $\mathcal{O}(n)$ for $n$ big enough by using Serre's theorem (\cite[Theorem II.5.17]{Ha}) on the special fibre. As a consequence one can write any formal metric on any line bundle on $X$ as a quotient of two semipositive formal metrics (on possibly different line bundles).
		We will see below, that $\mu_{m,k}(Z)$ is bounded with respect to $k$ for every compact subset $Z\subseteq V$. Hence, as the set of piecewise $\mathbb{Q}$-linear metrics is dense in the space of continuous metrics on $\mathcal{O}_{X^{\textup{an}}}$ with respect to uniform convergence (Proposition \ref{dense}), we may assume that $f=-\log\|1\|$ for a formal metric $\|\cdot\|$ on $\mathcal{O}_{X^{\textup{an}}}$. Then we can write $\|\cdot\|=\|\cdot\|_+/\|\cdot\|_-$ for two semipositive formal metrics $\|\cdot\|_+,\|\cdot\|_-$ on some line bundles $L_+$ respectively $L_-$ on $X^{\textup{an}}$. In fact $L_+=L_-$ but we will use the notation $\overline{L}_+$ and $\overline{L}_-$ to distinguish between the two metrics. Write $\overline{\mathcal{O}}_X^f$ for the line bundle $\mathcal{O}_{X^{\textup{an}}}\Big|_{V}$ endowed with the metric $\|1\|=e^{-f}$ and to shorten notation $\mu_m:=c_1(\overline{L}_{i_1,1})\wedge...\wedge c_1(\overline{L}_{i_m,1})$ which is a purely formal notation. Furthermore without loss of generality assume $i_1=1,...,i_m=m$. We have
		\begin{align*}
			\Big|&\int_{V} f\;\mu_{m,k} - \int_{V} f\;\mu_{m,k'}\Big| \\
			&= \Big|\sum_{i=1}^{n-m} \int_{V} f\; \mu_m\wedge c_1\left(\overline{L}_{m+1,k}\right)\wedge...\wedge c_1\left(\overline{L}_{m+i,k}\right)\wedge c_1\left(\overline{L}_{m+i+1,k'}\right)\wedge...\wedge c_1\left(\overline{L}_{n,k'}\right) \\
			&\phantom{\Big|\sum_{i=1}^{n}}- \int_{V} f\; \mu_m\wedge c_1\left(\overline{L}_{m+1,k}\right)\wedge...\wedge c_1\left(\overline{L}_{m+i-1,k}\right)\wedge c_1\left(\overline{L}_{m+i,k'}\right)\wedge...\wedge c_1\left(\overline{L}_{n,k'}\right)\Big| \\
			&= \Big|\sum_{i=1}^{n-m} \int_{V} f\; \mu_m\wedge...\wedge c_1\left(\overline{L}_{m+i-1,k}\right)\wedge c_1\left(\overline{L}_{m+i,1}\otimes\overline{\mathcal{O}}_{m+i,k}\right)\wedge c_1\left(\overline{L}_{m+i+1,k'}\right)\wedge... \\
			&\phantom{\Big|\sum_{i=1}^{n}}- \int_{V} f\; \mu_m\wedge...\wedge c_1\left(\overline{L}_{m+i-1,k}\right)\wedge c_1\left(\overline{L}_{m+i,1}\otimes\overline{\mathcal{O}}_{m+i,k'}\right)\wedge c_1\left(\overline{L}_{m+i+1,k'}\right)\wedge...\Big| \\
			&= \Big|\sum_{i=1}^{n-m} \int_{V} f\; \mu_m\wedge...\wedge c_1\left(\overline{L}_{m+i-1,k}\right)\wedge c_1\left(\overline{\mathcal{O}}_{m+i,k}\right)\wedge c_1\left(\overline{L}_{m+i+1,k'}\right)\wedge... \\
			&\phantom{\Big|\sum_{i=1}^{n}}- \int_{V} f\; \mu_m\wedge...\wedge c_1\left(\overline{L}_{m+i-1,k}\right)\wedge c_1\left(\overline{\mathcal{O}}_{m+i,k'}\right)\wedge c_1\left(\overline{L}_{m+i+1,k'}\right)\wedge...\Big| \\
		\end{align*}
		Since the support of $f$ is contained in $V$ and by Lemma \ref{local} these last integrals depend only on the restrictions of the metrics to $V$. Hence we may instead consider them as integrals over $X^{\textup{an}}$ which allows us to use Lemma \ref{switch} as $X^{\textup{an}}$ has no boundary (\cite[Theorem 3.4.1]{Be1}). In combination with an index shift, the last term amounts to
		\begin{align*}
			\Big|&\sum_{i=m+1}^{n} \int_{X^{\textup{an}}} -\log\|1\|'_{i,k}\; \mu_m\wedge...\wedge c_1\left(\overline{L}_{i-1,k}\right)\wedge c_1\left(\overline{\mathcal{O}}_X^f\right)\wedge c_1\left(\overline{L}_{i+1,k'}\right)\wedge... \\
			&- \int_{X^{\textup{an}}} -\log\|1\|'_{i,k'}\; \mu_m\wedge...\wedge c_1\left(\overline{L}_{i-1,k}\right)\wedge c_1\left(\overline{\mathcal{O}}_X^f\right)\wedge c_1\left(\overline{L}_{i+1,k'}\right)\wedge...\Big|
		\end{align*}
		As any point in $X^{\textup{an}}\setminus\supp(f)$ has a strictly $K$-analytic neighbourhood on which $f$ vanishes, the support of $\mu_m\wedge...\wedge c_1\left(\overline{L}_{i-1,k}\right)\wedge c_1\left(\overline{\mathcal{O}}_X^f\right)\wedge c_1\left(\overline{L}_{i+1,k'}\right)\wedge...$ is contained in $\supp(f)$ by Lemma \ref{local}. So the last display equals
		\begin{align*}
			\Big|&\sum_{i=m+1}^{n} \int_{\supp(f)} -\log\|1\|'_{i,k}\; \mu_m\wedge...\wedge c_1\left(\overline{L}_{i-1,k}\right)\wedge c_1\left(\overline{\mathcal{O}}_X^f\right)\wedge c_1\left(\overline{L}_{i+1,k'}\right)\wedge... \\
			&- \int_{\supp(f)} -\log\|1\|'_{i,k'}\; \mu_m\wedge...\wedge c_1\left(\overline{L}_{i-1,k}\right)\wedge c_1\left(\overline{\mathcal{O}}_X^f\right)\wedge c_1\left(\overline{L}_{i+1,k'}\right)\wedge...\Big| \\
			%&= \Big|\sum_{i=m+1}^{n} \int_{\supp(f)} (-\log\|1\|'_{i,k}+\log\|1\|'_{i,k'})\; ...\wedge c_1\left(\overline{L}_{i-1,k}\right)\wedge c_1\left(\hat{\mathcal{O}}_X^f\right)\wedge c_1\left(\overline{L}_{i+1,k'}\right)\wedge...\Big| \\
			&= \Big|\sum_{i=m+1}^{n} \int_{\supp(f)} \log\Big(\|1\|'_{i,k'}/\|1\|'_{i,k}\Big)\; ...\wedge c_1\left(\overline{L}_{i-1,k}\right)\wedge c_1\left(\overline{\mathcal{O}}_X^f\right)\wedge c_1\left(\overline{L}_{i+1,k'}\right)\wedge...\Big| \\
			&\leq 2\cdot\sum_{i=m+1}^n \sup_{x\in \supp(f)}\left|\log\Big(\|1\|'_{i,k}(x)/\|1\|'_{i,k'}(x)\Big)\right| \\
			&\phantom{\leq\sum_{i=m+}^n}\cdot\max_{s\in\{+,-\}}\mu_m\wedge...\wedge c_1\left(\overline{L}_{i-1,k}\right)\wedge c_1\left(\overline{L}_s\right)\wedge c_1\left(\overline{L}_{i+1,k'}\right)\wedge...(\supp(f)) \underset{k,k'\rightarrow\infty}{\longrightarrow} 0.
		\end{align*}
		Here the last term converges to zero as $\sup_{x\in\supp(f)}\left|\log\Big(\|1\|'_{i,k}(x)/\|1\|'_{i,k'}(x)\Big)\right|$ tends to zero by uniform convergence of $\|\cdot\|'_{i,k}$ and compactness of $\supp(f)$ and $\mu_m\wedge c_1\left(\overline{L}_{m+1,k}\right)\wedge...\wedge c_1\left(\overline{L}_{i-1,k}\right)\wedge c_1\left(\overline{L}_s\right)\wedge c_1\left(\overline{L}_{i+1,k'}\right)\wedge...\wedge c_1\left(\overline{L}_{n,k'}\right)$ are positive measures on $V$ which converge by the induction hypothesis weakly to a positive Radon measure which implies that their mass of $\supp(f)$ is bounded with respect to $k,k'$. To go into more detail, let $g$ be a continuous non-negative function on $V$ with compact support such that $g(x)>1$ for all $x\in \supp(f)$. The existence of such a function follows for example from a partition of unity argument (\cite[1.5.1]{F}) applied to the open cover $\{\overline{V}\setminus \supp(f), V\}$ of the closure $\overline{V}$ of $V$ (note that $\overline{V}$ is compact as $X$ is proper over $K$). Then
		\begin{align*}
			\mu_m&\wedge c_1\left(\overline{L}_{m+1,k}\right)\wedge...\wedge c_1\left(\overline{L}_{i-1,k}\right)\wedge c_1\left(\overline{L}_s\right)\wedge c_1\left(\overline{L}_{i+1,k'}\right)\wedge...\wedge c_1\left(\overline{L}_{n,k'}\right)(\supp(f)) \\
			&\leq\int g \;\mu_m\wedge c_1\left(\overline{L}_{m+1,k}\right)\wedge...\wedge c_1\left(\overline{L}_{i-1,k}\right)\wedge c_1\left(\overline{L}_s\right)\wedge c_1\left(\overline{L}_{i+1,k'}\right)\wedge...\wedge c_1\left(\overline{L}_{n,k'}\right)
		\end{align*}
		where the last term converges for $k,k'\rightarrow\infty$ and is hence bounded with respect to $k,k'$. \\
		We now define a positive linear functional on the space of continuous functions with compact support in $V$ by
		\begin{align*}
			C_c(V)&\rightarrow\mathbb{R}_{\geq 0}, \\
			f&\mapsto \lim_{k\rightarrow\infty}\int_V f \;\mu_{m,k}.
		\end{align*} 
		By the Riesz Representation Theorem (\cite[Theorem 2.14]{Ru}) this corresponds to a positive Radon measure $\mu$ on $V$ and we have $\mu_{m,k}\rightarrow \mu$ weakly for $k\rightarrow\infty$. \\
		It remains to show that $\mu_{m,k}(Z)$ is bounded with respect to $k$ for every compact subset $Z\subseteq V$. So let $Z\subseteq V$ be compact and $f$ a continuous non-negative function on $V$ with compact support such that $f(x)>1$ for all $x\in Z$. As above the existence of such a function follows from a partition of unity argument (\cite[1.5.1]{F}) applied to the open cover $\{\overline{V}\setminus Z, V\}$ of the closure $\overline{V}$ of $V$. Again we may assume that $f$ is a model function, i.e. of the from $-\log\|\cdot\|$ for a piecewise $\mathbb{Q}$-linear metric $\|\cdot\|$ on $\mathcal{O}_{X^{\textup{an}}}$ (we can even assume that $\|\cdot\|$ is a formal metric) and we use the same notation as above. To be more precise, let $\epsilon>0$ such that $f(x)>1+\epsilon$ for all $x\in Z$. First extend $f$ to $X^{\textup{an}}$ by zero and then define a new function $\tilde{f}$ by $\tilde{f}(x)=f(x)-\epsilon/2$. By Proposition \ref{dense} we may approximate $\tilde{f}$ by a model function $\phi$ such that $|\phi(x)-\tilde{f}(x)|<\epsilon/2$ for all $x\in X^{\textup{an}}$. Then by \cite[Proposition 2.12 (d)]{GM}, $\max\{0,\phi\}$ is a model function on $X^{\textup{an}}$ with compact support in $V$ which is greater than one at $Z$. We have
		\begin{align*}
			\sup_{k\in\mathbb{N}}\mu_{m,k}(Z)&\leq\sup_{k\in\mathbb{N}}\int_{V} f\;\mu_{m,k} \\
			&=\sup_{k\in\mathbb{N}}\int_{\supp(f)} f\;\mu_m\wedge c_1(\overline{L}_{m+1,k})\wedge...\wedge c_1(\overline{L}_{n,k}) \\
			&=\sup_{k\in\mathbb{N}}\int_{\supp(f)} f\;\mu_m\wedge c_1(\overline{L}_{m+1,1}\otimes\overline{\mathcal{O}}_{m+1,k})\wedge c_1(\overline{L}_{m+2,k})\wedge...\wedge c_1(\overline{L}_{n,k}) \\
			&=\sup_{k\in\mathbb{N}}\int_{\supp(f)} f\;\mu_m\wedge c_1(\overline{L}_{m+1,1})\wedge c_1(\overline{L}_{m+2,k})\wedge...\wedge c_1(\overline{L}_{n,k}) \\
			&\phantom{\lim_{k\rightarrow\infty}} + \int_{\supp(f)} f\;\mu_m\wedge c_1(\overline{\mathcal{O}}_{m+1,k})\wedge c_1(\overline{L}_{m+2,k})\wedge...\wedge c_1(\overline{L}_{n,k})
		\end{align*}
		Again using Lemma \ref{switch} and the same argumentation as above for the second summand this amounts to
		\begin{align*}
			\sup_{k\in\mathbb{N}}&\int_{\supp(f)} f\;\mu_m\wedge c_1(\overline{L}_{m+1,1})\wedge c_1(\overline{L}_{m+2,k})\wedge...\wedge c_1(\overline{L}_{n,k}) \\
			&\phantom{\lim_{k\rightarrow\infty}} + \int_{\supp(f)} -\log\|1\|'_{m+1,k}\;\mu_m\wedge c_1(\overline{\mathcal{O}}_X^{f})\wedge c_1(\overline{L}_{m+2,k})\wedge ...\wedge c_1(\overline{L}_{n,k}) \\
			&\leq \sup_{k\in\mathbb{N}} \sup_{x\in V} f(x) \cdot \mu_m\wedge c_1(\overline{L}_{m+1,1})\wedge c_1(\overline{L}_{m+2,k})\wedge...\wedge c_1(\overline{L}_{n,k})(\supp(f)) \\
			&\phantom{\lim_{k\rightarrow\infty}} + \sup_{k\in\mathbb{N}}\sup_{x\in\supp(f)} \left|\log\Big(\|1\|'_{m+1,k}(x)\Big)\right| \\
			&\phantom{\lim_{k\rightarrow\infty} + }\cdot 2\cdot\max_{s\in\{+,-\}} \mu_m\wedge c_1(\overline{L}_s)\wedge c_1(\overline{L}_{m+2,k})\wedge ...\wedge c_1(\overline{L}_{n,k})(\supp(f))
		\end{align*}
		By the induction hypothesis all measures appearing in this last term converge for $k\rightarrow\infty$. Hence the measure of $\supp(f)$ is bounded with respect to $k$. Furthermore $\sup_{x\in V} f(x)<\infty$ as $f$ has compact support in $V$ and $\sup_{x\in\supp(f)}\left|\log\Big(\|1\|'_{m+1,k}(x)\Big)\right|$ is bounded with respect to $k$ by uniform convergence of $\Big(\|\cdot\|'_{m+1,k}\Big)_{k\in\mathbb{N}}$ and compactness of $\supp(f)$. We conclude that the last term is bounded with respect to $k$. This proves the induction step. The claim is then the case $m=0$.
	\end{proof}
	\begin{Remark}
		\label{independent} In the situation of Proposition \ref{convergence}, the limit depends only on the metrics $\|\cdot\|_i$ but not on the sequences $(\|\cdot\|_{i,k})_{k\in\mathbb{N}}$. Namely, if $(\|\cdot\|'_{i,k})_{k\in\mathbb{N}}$ are other sequences converging uniformly to $\|\cdot\|_i$ then the sequences $(\|\cdot\|^{\prime \prime}_{i,k})_{k\in\mathbb{N}}$ defined by
		\[
			\|\cdot\|^{\prime \prime}_{i,k}:=\begin{cases} \|\cdot\|_{i,\frac{k}{2}}, \;k\text{ even} \\ \|\cdot\|'_{i,\frac{k-1}{2}}, \;k\text{ odd} \end{cases}
		\]
		converge uniformly to $\|\cdot\|_i$. As $(\|\cdot\|_{i,k})_{k\in\mathbb{N}}$ and $(\|\cdot\|'_{i,k})_{k\in\mathbb{N}}$ are subsequences of $\|\cdot\|^{\prime\prime}_{i,k}$ the limit of the measures is the same. We denote the measure corresponding to the metrics $\|\cdot\|_i$ by $c_1(\overline{L}_1)\wedge...\wedge c_1(\overline{L}_n)$.
	\end{Remark}
	\begin{Corollary}
		Let $X$ be a separated scheme of finite type over $K$ of dimension $n$. Let $M_1,...,M_n$ be line bundles on $X^{\textup{an}}$, $V$ an open subset of $X^{\textup{an}}$ and for $i\in\{1,...,n\}$ let $(\|\cdot\|_{i,k})_{k\in\mathbb{N}}$ be piecewise $\mathbb{Q}$-linear metrics on $M_1\Big|_V,...,M_n\Big|_V$ converging uniformly to a continuous metric $\|\cdot\|_i$ on $M_i\Big|_V$. Suppose that all $\|\cdot\|_{i,k}$ are semipositive in $V$. Write $\overline{M}_{i,k}:=\left(M_i,\|\cdot\|_{i,k}\right)$ and let $\overline{L}_1,...,\overline{L}_n$ be line bundles on $X^{\textup{an}}$ endowed with piecewise $\mathbb{Q}$-linear metrics on $V$.
		Then the measures $c_1\left(\overline{L}_1\otimes\overline{M}_{1,k}\right)\wedge ...\wedge c_1\left(\overline{L}_n\otimes\overline{M}_{n,k}\right)$ converge weakly to a Radon measure on $V$ denoted by $c_1(\overline{L}_1\otimes(M_1,\|\cdot\|_1))\wedge...\wedge c_1(\overline{L}_n\otimes(M_n,\|\cdot\|_n))$ (as above this measure does not depend on the choice of the $\|\cdot\|_{i,k}$).
	\end{Corollary}
	\begin{proof}
		As in the proof of Proposition \ref{convergence}, we may assume that $X$ is projective and write the metrics of the $L_i$ as a quotient of two semipositive metrics. Using multilinearity, this is now a direct consequence of Proposition \ref{convergence}.
	\end{proof}
	\begin{Remark}
		\label{extend} To extend the theory to the case where $K$ is not algebraically closed, choose an algebraic closure of $K$ and denote its completion by $\mathbb{C}_K$. Then we define the Monge-Ampère measure as the push-forward of the previously defined Monge-Ampère measure on the base change to $\mathbb{C}_K$. We explain it here in the situation of Definition \ref{measureopen}. Let $V$ be a strictly $K$-analytic Hausdorff space of dimension $n$, $\overline{L}_1,...,\overline{L}_n$ \textit{potentially} semipositive piecewise linear metrized line bundles on $V$ (i.e. metrized line bundles on $V$ which become semipositive piecewise linear metrized line bundles after base change to $\mathbb{C}_K$) and $\pi:V_{\mathbb{C}_K}\rightarrow V$ the base change. We can then define a measure on $V_{\mathbb{C}_K}$ with respect to the pull-backs of the line bundles $\overline{L}_1,...,\overline{L}_n$ by Definition \ref{measureopen} and push the resulting measure forward to $V$ via $\pi$. To make this well defined we show that $\pi$ is a proper map of topological spaces. So let $C\subseteq V$ be compact. Then we can cover $C$ by finitely many affinoid subdomains $U_1,...,U_r$. Then $\pi^{-1}(C)=\pi^{-1}\left(\bigcup C\cap U_i\right)=\bigcup\pi^{-1}(C\cap U_i)$ and it is enough to show that $\pi^{-1}(C\cap U_i)$ is compact for any $i$ so we may assume that $V$ is affinoid. But then $\pi$ is a continuous map between compact Hausdorff spaces and hence proper which yields the claim. We denote this measure again by $c_1(\overline{L}_1)\wedge...\wedge c_1(\overline{L}_n)$. One can check that all the results of this section remain true in this more general situation.
	\end{Remark}
	\begin{Def}
		\label{locallypotentiallysemipositive} Let $K$ be a complete, non-archimedean, non-trivially valued field, $V$ a strictly $K$-analytic space and $L$ a line bundle on $V$. A continuous metric $\|\cdot\|$ on $L$ is called \textit{locally semipositive} if for any $x\in V$ there is an open neighbourhood $U$ of $x$ such that $\|\cdot\|\Big|_U$ is a uniform limit of semipositive piecewise $\mathbb{Q}$-linear metrics on $L\Big|_U$. It is called \textit{locally potentially semipositive} if its base change to the completion of an algebraic closure of $K$ is locally semipositive. If $V$ is an open subset of $X^{\textup{an}}$ for a separated scheme $X$ of finite type over $K$ then using the Remarks \ref{independent} and \ref{extend} we define the Monge-Ampère measure $c_1(\overline{L}_1)\wedge...\wedge c_1(\overline{L}_n)$ for locally potentially semipositive metrized line bundles $\overline{L}_1,...,\overline{L}_n$ on $V$.
	\end{Def}
	\begin{Remark}
		The measures defined in this section are invariant under base change. In the spirit of Remark \ref{extend} this allows to define them in the trivially valued case for line bundles which become semipositive after base change to a non-trivially valued field. Such metrics and their measures are important for example in \cite{BJ}.
	\end{Remark}
	\section{Comparison of the real and non-archimedean Monge-Ampère operator}
	\label{sectioncomparison} In this section we want to compare the two measures introduced in the last section. In order to make sense of this, we start with a convex function $h$ on a closed face of some skeleton. Then one can associate to it a metric on the trivial line bundle which will turn out to be semipositive in the interior of the closed face. Thus we can associate to $h$ two measures, namely the real Monge-Ampère measure and the Chambert-Loir measure, sometimes also called the non-archimedean Monge-Ampère measure. In Corollary \ref{mainresult} we will see that they are equal up to scaling. In the following $K$ denotes a non-archimedean non-trivially valued field.
	\begin{Remark}
		\label{Notation} Let $\mathfrak{X}$ be a strongly nondegenerate strictly polystable formal scheme over $K^{\circ}$ of dimension $n+1$ with associated skeleton $\Delta$. Consider an $n$-dimensional closed face $\bar{\tau}$ of $\Delta$ with interior $\tau$ and the formal open subscheme $\mathfrak{X}'$ of $\mathfrak{X}$ consisting of all formal open subsets $\mathfrak{U}$ with $S(\mathfrak{U})=\bar{\tau}$. Let $h$ be a piecewise affine linear convex function (see Definition \ref{Defpiecewiseaffinelinear}) on $\bar{\tau}$ and $\mathfrak{D}$ a subdivision of $\bar{\tau}$ such that $h\Big|_{\Delta'}$ is affine linear for all $\Delta'\in\mathfrak{D}$. Let $\iota:\mathfrak{X}''\rightarrow\mathfrak{X}'$ be the corresponding formal scheme (cf. Construction \ref{Formalmodel}). We have seen in Proposition \ref{CartierDiv} that $h$ induces a Cartier divisor $D$ on $\mathfrak{X}''$. We set $\mathcal{O}(h\circ p_{\mathfrak{X}'}):=\mathcal{O}(D)$ where $p_{\mathfrak{X}'}:\mathfrak{X}^{\prime an}\rightarrow\overline{\tau}$ is the restriction of the contraction $p_{\mathfrak{X}}:\mathfrak{X}^{\textup{an}}\rightarrow\Delta$. For a line bundle $\mathfrak{L}$ on a formal scheme, we will denote by $c_1(\mathfrak{L})$ the first Chern class of the special fibre of $\mathfrak{L}$.
	\end{Remark}
	\begin{Theorem}
		\label{MAvertex} In the situation of Remark \ref{Notation} let $u\in\tau$ be a vertex of $\mathfrak{D}$ with corresponding irreducible component $Y\subseteq\tilde{\mathfrak{X}}''$ as in Corollary \ref{Corstratumface} (f). Let $S$ be the closed point in the special fibre of $\mathfrak{X}$ corresponding to $\tau$. Then
		\[
			\Deg\left(c_1\left(\mathcal{O}(h\circ p_{\mathfrak{X}'})\right)^n.Y\right)=\Deg(S)\cdot n!\cdot\MA(h)(u).
		\]
	\end{Theorem}
	\begin{proof}
		Note that $S$ is a closed point of $\tilde{\mathfrak{X}}'$ and hence proper over $\tilde{K}$. Therefore also $Y$ is proper over $\tilde{K}$ since it is a closed subset of $\tilde{\iota}^{-1}(S)$ and $\iota$ is proper by \cite[Corollary 4.4]{T1}. Let $D$ be the Cartier divisor on $\mathfrak{X}''$ induced by $h$ as in Proposition \ref{CartierDiv} such that $c_1\left(\mathcal{O}(h\circ p_{\mathfrak{X}'})\right)^{n}.Y=D^{n}.Y$. 
		We show by induction that for all $0\leq l\leq n$ there is a strata cycle $Y_l$ of dimension $n-l$ whose components are contained in $Y$ such that $\Deg(D^n.Y)=\Deg(D^{n-l}.Y_l)$. The case $l=0$ is clear by taking $Y_0:=Y$. Now let $l<n$ and $Y_l$ be as claimed. Let $Y'$ be a stratum of $Y_l$, such that $Y'$ is associated to an $l$-dimensional open face $\tau'$ of $\mathfrak{D}$, i.e. $Y'=\red_{\mathfrak{X}''}(p_{\mathfrak{X}''}^{-1}(\tau'))$ with $u\in\overline{\tau'}$ by the stratum face correspondence (Proposition \ref{Stratumface}). Using $\tau'\subseteq\tau\subseteq\mathbb{R}^n$, there is an affine linear function $a:\mathbb{R}^n\rightarrow\mathbb{R}$ such that $h\Big|_{\tau'}=a\Big|_{\tau'}$. Then $h-a\Big|_{\tau}$ defines a Cartier divisor $D_{Y'}$ on $\mathfrak{X}''$ by Proposition \ref{CartierDiv} which is numerically equivalent to $D$ on $Y$ by Lemma \ref{numericallyequivalent} and which is trivial on $Y'$ because $h-a\Big|_{\tau'}=0$. 
		Hence, as $\overline{Y'}$ is a strata subset, $D_{Y'}.\overline{Y'}$ is a strata cycle. Write $Y_l=\sum_{Y'}m_{Y'}\overline{Y'}$ where the sum ranges over a finite number of $n-l$-dimensional strata of $\tilde{\mathfrak{X}}''$ contained in $Y$. Then we can calculate:
		\begin{align*}
			\Deg(D^n.Y)&=\Deg\left(D^{n-l}.Y_l\right) \\
			&=\Deg\left(D^{n-l}.\sum_{Y'} m_{Y'} \overline{Y'}\right) \\
			&=\Deg\left(\sum_{Y'} m_{Y'} D^{n-l}.\overline{Y'}\right) \\
			&=\Deg\left(\sum_{Y'} m_{Y'} D^{n-l-1}.(D_{Y'}.\overline{Y'})\right) \\
			&=\Deg\left(D^{n-l-1}.\sum_{Y'} m_{Y'} D_{Y'}.\overline{Y'}\right)
		\end{align*}
		and $Y_{l+1}:=\sum_{Y'} m_{Y'} D_{Y'}.\overline{Y'}$ is a strata cycle as claimed. 
		We use this for $l=n$ to see that $\Deg(D^n.Y)=\Deg(Y_n)$ for a strata cycle $Y_n$ of dimension $0$ contained in $Y$. Its components are strata points $S_i$ of $\mathfrak{X}''$ which are mapped by $\iota$ to the point $S$ corresponding to $\tau$. Now let $\mathfrak{U}'\subseteq\mathfrak{X}'$ be a formal open subset with an étale morphism $\psi:\mathfrak{U}'\rightarrow\mathfrak{X}(\boldsymbol{n},\boldsymbol{a})$ such that $S$ is the distinguished stratum of $\mathfrak{U}'$ (cf. Proposition \ref{Gub2 Proposition 5.2}) and define $\mathfrak{U}'':=\iota^{-1}(\mathfrak{U}')$. Note that there is no factor $\mathfrak{X}(m)$ because $\tau$ is of maximal dimension. As the strata occurring in the intersection process correspond to open faces of $\mathfrak{D}$ with vertex $u$, their intersection with $\mathfrak{U}''$ is nonempty. Hence we may calculate the multiplicities of $Y_{n}$ locally on $\mathfrak{U}''$. The stratification of $\tilde{\mathfrak{U}}''$ is obtained by the preimages of the strata of $\tilde{\mathfrak{X}}(\boldsymbol{n},\boldsymbol{a})'$ (see proof of Proposition \ref{Stratumface}) with respect to the base change $\psi':\mathfrak{U}''\rightarrow\mathfrak{X}(\boldsymbol{n},\boldsymbol{a})'$ of $\psi$ (cf. Construction \ref{Formalmodel}). Let $Y_u=\overline{\tilde{\psi}'(\tilde{\mathfrak{U}}''\cap Y)}$ be the irreducible component in $\tilde{\mathfrak{X}}(\boldsymbol{n},\boldsymbol{a})'$ corresponding to $u$ and $D_u$ the Cartier divisor on $\mathfrak{X}(\boldsymbol{n},\boldsymbol{a})'$ whose pullback gives the Cartier divisor $D$ associated to $h$ on $\mathfrak{U}''$ (cf. proof of Proposition \ref{CartierDiv}). By applying the modifications of $D$ in the induction step also to $D_u$ we obtain a strata cycle $Y_{n}^t=\sum m_jP_j$ of $\tilde{\mathfrak{X}}(\boldsymbol{n},\boldsymbol{a})'$ whose pullback is $Y_{n}$ (as the intersection product is compatible with flat pullback by \cite[Proposition 2.3(d)]{Fu2}) and which has the same degree as $D_u^n.Y_u$ using Lemma \ref{numericallyequivalent}.
Now let 
		\begin{align*}
			\val:(\mathbb{G}_m^{\boldsymbol{n}})_K^{\textup{an}}&\rightarrow\mathbb{R}^{\boldsymbol{n}}, \\ 
			q&\mapsto(-\log q(x_{01}),...,-\log q(x_{0n_0}),...,-\log q(x_{p1}),...,-\log q(x_{pn_p}))
		\end{align*}
		and $\Sigma:=\left\{\boldsymbol{w}\in\mathbb{R}_{\geq 0}^{\boldsymbol{n}}\;\Big|\;w_{i1}+...+w_{in_i}\leq v(a_i),0\leq i \leq p\right\}$. As we have an isomorphism 
		\[
			\mathfrak{X}(\boldsymbol{n},\boldsymbol{a})^{\textup{an}}\tilde{\rightarrow}\val^{-1}(\Sigma)
		\]
		and using \cite[Corollary 6.15]{Gub3}, we find that $Y_u$ is a toric variety with fan given by the cones generated by $\Delta'-u$ for $\Delta'\in\mathfrak{D}$ with vertex $u$ (in fact we identify $\tau$ with $\Sigma$ by forgetting about the coordinate with index $0$ for each $i$). $D_u\Big|_{Y_u}$ is given up to multiplication by a constant by the divisor $D'_u$ on $Y_u$ associated to the linear function $h':=h(\cdot+u)-h(u)$. By \cite[3.4,5.3]{Fu1} we have
		\[
			\lambda(P_{D'_u})=\frac{\Deg({D'}_u^n.Y_u)}{n!},
		\] 
		where
		\[
			P_{D'_u}=\left\{y\in\mathbb{R}^n\;\Big|\;\langle z,y\rangle\leq\psi_{D'_u}(z)=h'(z)\;\forall z\in\mathbb{R}^n\right\}=\nabla h'(0)		
		\]
		and $\lambda$ denotes the standard Lebesgue measure. For the last term we get
		\begin{align*}
			\nabla h'(0)&=\left\{p\in\mathbb{R}^n\;\Big|\;\forall x\in\tau-u\;:\;h'(0)+\langle x,p\rangle\leq h'(x)\right\} \\
			&=\left\{p\in\mathbb{R}^n\;\Big|\;\forall x\in\tau-u\;:\;\langle x,p\rangle\leq h(x+u)-h(u)\right\} \\
			&=\left\{p\in\mathbb{R}^n\;\Big|\;\forall x\in\tau\;:\;h(u)+\langle x-u,p\rangle\leq h(x)\right\} \\
			&=\nabla h(u).
		\end{align*}
		Hence
		\[
			\frac{1}{n!}\Deg(Y_n^t)=\frac{\Deg({D'}_u^n.Y_u)}{n!}=\lambda(P_{D'_u})=\lambda(\nabla h(u))=\MA(h)(\{u\}).
		\]
		With $\iota'$ denoting the morphism $\mathfrak{X}(\boldsymbol{n},\boldsymbol{a})'\rightarrow\mathfrak{X}(\boldsymbol{n},\boldsymbol{a})$ we conclude
		\[
			\Deg(D^n.Y)=\Deg\left(Y_{n}\right)=\Deg\left(\psi^{\prime \ast}Y_n^t\right)=\Deg\left(\iota_\ast\psi^{\prime \ast}\sum m_jP_j\right).
		\]
		Using \cite[Proposition 1.7]{Fu2} this equals
		\[
			\Deg\left(\psi^{\ast}\iota'_{\ast}\sum m_jP_j\right)=\Deg\left(\psi^{\ast}\sum m_j[P_j:\{\tilde{\boldsymbol{0}}\}]\cdot\{\tilde{\boldsymbol{0}}\}\right).
		\]
		As $\psi^{-1}(\{\tilde{\boldsymbol{0}}\})=S$ is reduced since $\psi$ is smooth, this amounts to
		\[
			\Deg\left(\sum m_j\Deg(P_j)S\right)=\Deg(S)\cdot\Deg(Y_n^t)=\Deg(S)\cdot n!\cdot\MA(h)(\{u\}).
		\]
		This yields the equality we wanted to prove.
	\end{proof}
	\begin{Remark}
		Using the same arguments, one can show the following more general formula: In the situation of Theorem \ref{MAvertex} instead of only one function $h$ consider $h_1,...,h_n$ piecewise affine linear convex functions on $\bar{\tau}$. Refine the subdivision $\mathfrak{D}$ such that it suits every $h_i$. Then
		\[
			\Deg\left(\bigwedge_{i=1}^n c_1\left(\mathcal{O}(h_i\circ p_{\mathfrak{X}'})\right).Y \right)=\Deg(S)\cdot n!\cdot\MA(h_{1},...,h_{n})(u),
		\]
		where 
		\[
			\MA(h_{1},...,h_{n}):=\frac{1}{n!}\sum_{k=1}^n (-1)^{n-k}\cdot\sum_{1\leq i_1<...<i_k\leq n}\MA(h_{i_1}+...+h_{i_k})
		\] 
		denotes now the mixed Monge-Ampère measure of $h_{1},...,h_{n}$ (for details see \cite[§5]{PR}). 
	\end{Remark}
	\begin{Remark}
		In the situation of Theorem \ref{MAvertex} we denote by $\overline{\mathcal{O}}^{h\circ p_{\mathfrak{X}'}}$ the trivial line bundle on $\mathfrak{X}'^{\textup{an}}$ together with the metric which is given by $\|1\|=e^{-h\circ p_{\mathfrak{X}'}}$. After base change to the completion of an algebraic closure $\mathbb{C}_K$ of $K$ this becomes a formally metrized line bundle by Proposition \ref{CartierDiv}. So similarly as in Remark \ref{extend} we can define its non-archimedean Monge-Ampère measure by base change to $\mathbb{C}_K$.
	\end{Remark}
	\begin{Corollary}
		\label{Comparison}
		We have
		\[
			c_1\left(\overline{\mathcal{O}}^{h\circ p_{\mathfrak{X}'}}\right)^n=\Deg(S)\cdot n!\cdot\MA(h)
		\]
		on $p_{\mathfrak{X}'}^{-1}(\tau)$, where $\MA(h)$ is understood to be a measure on $\mathfrak{X}'^{\textup{an}}$ by pushforward with the inclusion $\tau\hookrightarrow \mathfrak{X}'^{\textup{an}}$.
	\end{Corollary}
	\begin{proof}
		 We already know by Theorem \ref{MAvertex} that the equation holds on the set of vertices. Furthermore it is clear from the definition, that $c_1\left(\overline{\mathcal{O}}^{h\circ p_{\mathfrak{X}'}}\right)^n$ is supported on the vertices of $\mathfrak{D}$. What remains to show is that this also holds for $\MA(h)$.
		 
		Let $U:=\tau\setminus\left\{u\in\tau\;\Big|\;u\text{ is a vertex of }\mathfrak{D}\right\}$. We want to show $\MA(h)(U)=0$. Let $\Delta_1,...,\Delta_r$ be the open faces of $\mathfrak{D}$ of dimension at least one. For every $j\in\{1,...,r\}$ there is a $v_j\in\mathbb{R}^n\setminus\{0\}$ such that for all $y\in\Delta_j$ there exists $\epsilon\in\mathbb{R}_+$ such that $y\pm\epsilon v_j\in\Delta_j$. Furthermore $h_j:=h\Big|_{\Delta_j}=\boldsymbol{m}_j\boldsymbol{x}+v(\alpha_j)$ for some $\boldsymbol{m}_j\in\mathbb{Z}^n$ and $\alpha_j\in K^\times$ and we define $h_j^{lin}:=\boldsymbol{m}_j\boldsymbol{x}$. Now let $y\in U$. Then there is an $i$ such that $y\in\Delta_i$. For $p\in\nabla h(y)$ and $\epsilon$ as above it follows
		\begin{align*}
			\epsilon\langle v_i,p\rangle &=h_i(y)+\langle y+\epsilon v_i -y,p\rangle -h_i(y) \\
			&\leq h_i(y+\epsilon v_i)-h_i(y) \\
			&=h_i^{lin}(\epsilon v_i) \\
			&=\epsilon h_i^{lin}(v_i),
		\end{align*}
		hence
		\[
			\langle v_i,p\rangle\leq h_i^{lin}(v_i).
		\]
		A similar argument shows 
		\[
			-\epsilon\langle v_i,p\rangle\leq -\epsilon h_i^{lin}(v_i)
		\]
		and hence
		\[
			\langle v_i,p\rangle \geq h_i^{lin}(v_i).
		\]
		We conclude $\langle v_i,p\rangle=h_i^{lin}(v_i)$ and $p$ lies in a hypersurface which depends on $i$ but not on $y$. Hence $\bigcup_{y\in U} \nabla h(y)$ is contained in the union of $r$ hypersurfaces. Therefore 
		\[
			\MA(h)(U)=\lambda\left(\bigcup_{y\in U}\nabla h(y)\right)=0,
		\]
		\end{proof}
		In the following we consider a proper algebraic variety $X$ over $K$ of dimension $n$.
		\begin{Proposition}
			\label{semipositive} Let $\mathfrak{X}$ be a strongly nondegenerate strictly polystable formal model of $X^{\textup{an}}$ over $K^\circ$ with associated skeleton $\Delta$, $\tau$ an $n$-dimensional open face of $\Delta$ and $h$ a rational piecewise affine linear convex function on $\tau$. Then the metric on $\mathcal{O}_{X^{\textup{an}}}\Big|_{p_{\mathfrak{X}}^{-1}(\tau)}$ given by $\|1\|=e^{-h\circ p_{\mathfrak{X}}}$ is a semipositive piecewise $\mathbb{Q}$-linear metric. 
		\end{Proposition}
		\begin{proof} 
			Let $y\in p_{\mathfrak{X}}^{-1}(\tau)$ and $x:=p_{\mathfrak{X}}(y)\in\tau$. There is an open neighbourhood $U$ of $x$ in $\tau$ such that we can write $h\Big|_U=\max_{i=1,...,s} h_i\Big|_U$ for suitable rational affine linear functions $h_i$ on $\tau$. After passing to some multiple, each $h_i$ induces a formal metric on $\mathfrak{X}^{\prime an}$ by Proposition \ref{CartierDiv} where $\mathfrak{X}'$ is defined as in Remark \ref{Notation}. Therefore the $h_i$ induce piecewise $\mathbb{Q}$-linear metrics on $\mathcal{O}_{\mathfrak{X}^{\textup{an}}}\Big|_{p_{\mathfrak{X}}^{-1}(\tau)}$ since $p_{\mathfrak{X}}^{-1}(\tau)\subseteq\mathfrak{X}^{\prime an}$. Hence in the neighbourhood $p_{\mathfrak{X}}^{-1}(U)$ of $y$, the metric induced by $h$ is given as the minimum of the metrics corresponding to the $h_i$, which are semipositive at $y$ by Lemma \ref{numericallyequivalent}. Indeed let $(\mathfrak{X}''_i,\mathfrak{L}_i)$ be a formal model of the trivial bundle associated to $h_i$ as obtained by Proposition \ref{CartierDiv}. Then by \cite[Proposition 6.5]{GK1} (the proof of the implication we need does neither use that $K$ is algebraically closed nor that the generic fibre is algebraic) it is enough to show that $\Deg_{\mathfrak{L}_i}(Y)\geq 0$ for any closed curve $Y$ in $\tilde{\mathfrak{X}}''_i$ with $Y\subseteq\red(p_{\mathfrak{X}''_i}^{-1}(\tau))$ but by Lemma \ref{numericallyequivalent} we even have equality. Now we extend the metrics induced by the $h_i$ from a compact strictly $K$-analytic neighbourhood of $y$ to $X^{\textup{an}}$ by \cite[Proposition 2.7]{GM} and then it follows from Proposition \ref{GM Proposition 3.12} that $||\cdot||$ is semipositive at $y$.
		\end{proof}
		\begin{Corollary}
			\label{mainresult}
			Let $\mathfrak{X}$ be a strongly nondegenerate strictly polystable formal model of $X^{\textup{an}}$ over $K^\circ$ with associated skeleton $\Delta$. Let $\tau$ be an $n$-dimensional open face of $\Delta$ and $h$ a convex function on $\tau$. Denote by $\overline{\mathcal{O}}^{h\circ p_{\mathfrak{X}}}$ the trivial bundle on $p_{\mathfrak{X}}^{-1}(\tau)$ endowed with the metric given by $\|1\|=e^{-h\circ p_{\mathfrak{X}}}$. Then the latter is locally a semipositive metric (Definition \ref{locallypotentiallysemipositive}) and
			\[
				c_1\left(\overline{\mathcal{O}}^{h\circ p_{\mathfrak{X}}}\right)^n=\Deg(S)\cdot n!\cdot\MA(h)
			\]
			on $p_{\mathfrak{X}}^{-1}(\tau)$ where $S$ is the point in the special fibre of $\mathfrak{X}$ corresponding to $\tau$.
		\end{Corollary}
		\begin{proof}
			We can cover $\tau$ by polytopes $(\Delta_m)_{m\in\mathbb{N}}$ such that $\Delta_{m-1}\subseteq\Delta_m$. By \cite[Proposition 2.5.24]{BPS} for each $m$ there is a family of rational piecewise affine linear convex functions $(h_i^m)_{i\in\mathbb{N}}$ on $\Delta_m$ converging uniformly to $h\Big|_{\Delta_m}$ (note that after normalization we can assume that $\mathbb{Z}$ is contained in the value group of $K$). We extend these functions to rational piecewise affine linear convex functions on $\overline{\tau}$. Then by Proposition \ref{semipositive} the metrics induced by the $h_i^m$ are semipositive piecewise $\mathbb{Q}$-linear metrics on $p_{\mathfrak{X}}^{-1}(\tau)$ which implies that the metric induced by $h\Big|_{\Delta_m}$ is semipositive. By Corollary \ref{Comparison} we have
			\[
				c_1\left(\overline{\mathcal{O}}^{h_i^m\circ p_{\mathfrak{X}}}\right)^n=\Deg(S)\cdot n!\cdot\MA(h_i^m)
			\]
			for every $m,i\in\mathbb{N}$. Denoting the interior of $\Delta_m$ by $\Delta_m^\circ$ and using Proposition \ref{convergence} we find that for fixed $m$ the left hand side converges to $c_1\left(\overline{\mathcal{O}}^{h\circ p_{\mathfrak{X}}}\right)^n$ on $p_{\mathfrak{X}}^{-1}(\Delta_m^\circ)$. The right hand side converges to $\Deg(S)\cdot n!\cdot\MA(h)$ on $\Delta_m^\circ$ by continuity of the real Monge-Ampère operator. As this holds for any $m$ and the $\Delta_m$ cover $\tau$ this proves the corollary.
		\end{proof}
		\begin{Def}
			\label{convex} Let $\mathfrak{X}$ be a strongly nondegenerate polystable formal scheme with associated skeleton $\Delta$ and $\tau$ an open face of $\Delta$. A function $h:\tau\rightarrow\mathbb{R}$ is called convex if there exists a surjective étale morphism $\varphi:\mathfrak{X}'\rightarrow\mathfrak{X}$ with a strongly nondegenerate strictly polystable formal scheme $\mathfrak{X}'$ and an open face $\tau'$ of the skeleton associated to $\mathfrak{X}'$ with $\varphi^{\textup{an}}(\tau')=\tau$ such that $h\circ\varphi^{\textup{an}}:\tau'\rightarrow\mathbb{R}$ is convex. For such a convex function $h$ on $\tau$ we define $\MA(h):=\left(\varphi^{\textup{an}}\Big|_{p_{\mathfrak{X}'}^{-1}(\tau')}\right)_{\ast}\MA\left(h\circ\varphi^{\textup{an}}\Big|_{\tau'}\right)$. It will follow from Corollary \ref{mainpolystable} that this is independent of the choices.
		\end{Def}
		\begin{Proposition}
			\label{semipositiveII} Let $K$ be algebraically closed, $\mathfrak{X}$ a strongly nondegenerate polystable formal model of $X^{\textup{an}}$ over $K^\circ$ with associated skeleton $\Delta$, $\tau$ an $n$-dimensional open face of $\Delta$ and $h$ a rational piecewise affine linear convex function on $\tau$. Then the metric on $\mathcal{O}_{X^{\textup{an}}}\Big|_{p_{\mathfrak{X}}^{-1}(\tau)}$ which is given by $\|1\|=e^{-h\circ p_{\mathfrak{X}}}$ is a semipositive piecewise $\mathbb{Q}$-linear metric.
		\end{Proposition}
		\begin{proof}
			Let $\mathfrak{X}'$ be a strongly nondegenerate strictly polystable formal scheme such that there is a surjective étale morphism $\varphi:\mathfrak{X}'\rightarrow\mathfrak{X}$. Let $q\in\tilde{\mathfrak{X}}$ be the closed point corresponding to $\tau$. By Proposition \ref{SFCBer} we have $\red_{\mathfrak{X}}^{-1}(q)=p_{\mathfrak{X}}^{-1}(\tau)$. Choose $q'\in\tilde{\mathfrak{X}'}$ with $\varphi(q')=q$. By \cite[Proposition 2.9]{Gub1} we have that $\varphi$ induces an isomorphism $\red_{\mathfrak{X}'}^{-1}(q')\tilde{\rightarrow}p_{\mathfrak{X}}^{-1}(\tau)$. Hence the pullback of $\left(\mathcal{O}_{X^{\textup{an}}}\Big|_{p_{\mathfrak{X}}^{-1}(\tau)},\|\cdot\|\right)$ is the trivial bundle on $\red_{\mathfrak{X}'}^{-1}(q')$ endowed with the metric $\|1\|'=e^{-h\circ p_{\mathfrak{X}}\circ\varphi}$. Since $p_{\mathfrak{X}}\circ\varphi=\varphi\circ p_{\mathfrak{X}'}$ it follows that $\|\cdot\|'$ is the metric associated to the function $h\circ\varphi$ on $\tau':=p_{\mathfrak{X}'}(\red_{\mathfrak{X}'}^{-1}(q'))$ which is again rational piecewise affine linear by \cite[Theorem 6.1.1]{Be3} and we may assume it is convex by definition.
			Let $y\in p_{\mathfrak{X}}^{-1}(\tau)$.  
			In a neighbourhood of $p_{\mathfrak{X}'}(y')$ where $y'\in p_{\mathfrak{X}'}^{-1}(\tau')$ with $\varphi^{\textup{an}}(y')=y$ we can write $h\circ\varphi^{\textup{an}}=\max_{i=1,...,s}h'_i$ for suitable affine linear functions $h'_i$ on $\tau'$. Now as $\varphi^{\textup{an}}:p_{\mathfrak{X}'}^{-1}(\tau')\rightarrow p_{\mathfrak{X}}^{-1}(\tau)$ is an isomorphism we have $h=\max_{i=1,...,s} h_i$ where $h_i$ are the piecewise affine linear functions on $\tau$ satisfying $h'_i=h_i\circ\varphi^{\textup{an}}$. Now the metrics associated to the $h'_i$ are piecewise $\mathbb{Q}$-linear and semipositive in $y'$ by the same argument as in the proof of Proposition \ref{semipositive}. Hence the piecewise $\mathbb{Q}$-linear metrics associated to the $h_i$ extend from a compact strictly $K$-analytic neighbourhood of $y$ to global metrics by \cite[Proposition 2.7]{GM} which are semipositive in $y$. Now as $\|\cdot\|$ is locally around $y$ given as the minimum of these metrics, also $\|\cdot\|$ is a piecewise $\mathbb{Q}$-linear metric which is semipositive in $y$ by Proposition \ref{GM Proposition 3.12}.
		\end{proof}
		\begin{Corollary}
			\label{mainpolystable} Let $\mathfrak{X}$ be a strongly nondegenerate polystable formal model of $X^{\textup{an}}$ over $K^\circ$ with associated skeleton $\Delta$. Let $\tau$ be an $n$-dimensional open face of $\Delta$ with corresponding point $S$ in the special fibre of $\mathfrak{X}$ and $h$ a convex function on $\tau$. Denote by $\overline{\mathcal{O}}^{h\circ p_{\mathfrak{X}}}$ the trivial bundle on $p_{\mathfrak{X}}^{-1}(\tau)$ endowed with the metric given by $\|1\|=e^{-h\circ p_{\mathfrak{X}}}$. Then $\|\cdot\|$ is locally a potentially semipositive metric and
			\[
				c_1\left(\overline{\mathcal{O}}^{h\circ p_{\mathfrak{X}}}\right)^n=\Deg(S)\cdot n!\cdot\MA(h)
			\]
			on $p_{\mathfrak{X}}^{-1}(\tau)$. 
		\end{Corollary}
		\begin{proof}
			Let $\mathbb{C}_K$ be the completion of an algebraic closure of $K$. Then there are exactly $\Deg(S)$ points in the special fibre of $\mathfrak{X}_{\mathbb{C}_K}$ mapping to $S$, hence there are precisely $\Deg(S)$ open faces in the skeleton associated to $\mathfrak{X}_{\mathbb{C}_K}$ lying over $\tau$. As the base change induces an isomorphism of each of these faces with $\tau$, we have $\iota_{\ast}\MA(\iota^{\ast}h)=\Deg(S)\MA(h)$. 
			Using this and the invariance of the non-archimedean Monge-Ampère measure under base change we may assume $K=\mathbb{C}_K$. As in the proof of Proposition \ref{semipositiveII} we choose a strongly nondegenerate strictly polystable formal scheme $\mathfrak{X}'$ and a surjective étale morphism $\varphi:\mathfrak{X}'\rightarrow\mathfrak{X}$. Let $\tau'$ be an open face of the skeleton associated to $\mathfrak{X}'$ lying over $\tau$. As we have seen, $\varphi$ induces an isomorphism $p_{\mathfrak{X}'}^{-1}(\tau')\tilde{\rightarrow}p_{\mathfrak{X}}^{-1}(\tau)$. As in the proof of Corollary \ref{mainresult} there is a sequence of rational piecewise affine linear convex functions $(h'_i)_{i\in\mathbb{N}}$ on $\tau'$ converging locally uniformly to $h\circ\varphi^{\textup{an}}$. Let $h_i$ be the piecewise affine linear functions on $\tau$ such that $h_i\circ\varphi^{\textup{an}}=h'_i$. By Proposition \ref{semipositiveII} the metrics induced by the $h_i$ are semipositive piecewise $\mathbb{Q}$-linear metrics on $p_{\mathfrak{X}}^{-1}(\tau)$ which implies that the metric induced by $h$ is locally semipositive. As the restriction of $\varphi$ to $p_{\mathfrak{X}'}^{-1}(\tau')$ is an isomorphism onto $p_{\mathfrak{X}}^{-1}(\tau)$ we have
			\[
				c_1\left(\overline{\mathcal{O}}^{h_i\circ p_{\mathfrak{X}}}\right)^n=\left(\varphi\Big|_{p_{\mathfrak{X}'}^{-1}(\tau')}\right)_{\ast} c_1\left(\left(\varphi\Big|_{p_{\mathfrak{X}'}^{-1}(\tau')}\right)^{\ast}\overline{\mathcal{O}}^{h_i\circ p_{\mathfrak{X}}}\right)^n
			\]
			%As $X^{an}$ has no boundary and $p_{\mathfrak{X}'}^{-1}(\tau')$ is isomorphic to $p_{\mathfrak{X}}^{-1}(\tau)$, it does not contain any boundary points of $\mathfrak{X}'^{an}$. 
			By Corollary \ref{Comparison} we have
			\[
				c_1\left(\left(\varphi\Big|_{p_{\mathfrak{X}'}^{-1}(\tau')}\right)^{\ast}\overline{\mathcal{O}}^{h_i\circ p_{\mathfrak{X}}}\right)^n=c_1\left(\overline{\mathcal{O}}^{h_i\circ\varphi^{\textup{an}}\circ p_{\mathfrak{X}'}}\right)^n=n!\cdot\MA\left(h_i\circ\varphi^{\textup{an}}\Big|_{\tau'}\right).
			\]
			Hence
			\[
				c_1\left(\overline{\mathcal{O}}^{h_i\circ p_{\mathfrak{X}}}\right)^n=\left(\varphi\Big|_{p_{\mathfrak{X}'}^{-1}(\tau')}\right)_{\ast}\left(n!\cdot\MA\left(h_i\circ\varphi^{\textup{an}}\Big|_{\tau'}\right)\right)=n!\cdot\MA(h_i).
			\]
			It is easily seen that in Proposition \ref{convergence} we can replace uniform convergence by locally uniform convergence. The claim follows from this fact and continuity of the real Monge-Ampère operator.  
		\end{proof}
		\section{Applications to regularity}
		\label{application} In this section we use the connection of the non-archimedean Monge-Ampère operator to the real one to transfer two known regularity results for the solutions of the real Monge-Ampère equation to the non-archimedean case. Again $K$ denotes a non-archimedean non-trivially valued field.
		\begin{Def}
			Let $\Omega\subseteq\mathbb{R}^n$ be an open subset and $k\in\mathbb{N}$. We write $C^k(\Omega)$ for the space of real valued, $k$ times continuously differentiable functions on $\Omega$.
			Furthermore we denote by $L_{loc}^1(\Omega)$ the space of locally integrable functions on $\Omega$ i.e. functions $f:\Omega\rightarrow\mathbb{R}$ such that the restriction of $f$ to any compact subset of $\Omega$ is integrable. Let $f,g\in L_{loc}^1(\Omega)$ and $\beta\in\mathbb{N}^n$. We say that $g$ is the $\beta$-th weak derivative of $f$ if for any test function $\varphi\in C^\infty(\Omega)$ with compact support we have
			\[
				\int_{\Omega}fD^\beta\varphi\;\boldsymbol{dx}=(-1)^{|\beta|}\int_{\Omega}g\varphi\;\boldsymbol{dx}
			\]
			where $\boldsymbol{dx}$ denotes the Lebesgue measure on $\mathbb{R}^n$. We denote by $W^{k,1}_{loc}(\Omega)$ the space of locally integrable functions on $\Omega$ whose weak derivatives exist up to order $k$.
		\end{Def}
		\begin{Proposition}
			Let $X$ be an $n$-dimensional proper variety over $K$ and $\overline{L}$ a line bundle with a fixed formal metric. Let $\mu$ be a positive Borel measure on $X^{\textup{an}}$ and $\varphi$ a continuous function on $X^{\textup{an}}$ such that the metric on $\overline{L}\otimes\overline{\mathcal{O}}^\varphi$ is semipositive and solving the equation
			\[
				c_1(\overline{L} \otimes \overline{\mathcal{O}}^\varphi)^n=\mu.
			\]
			Let $\tau$ be an $n$-dimensional open face of some skeleton $\Delta$ associated to a strongly nondegenerate strictly polystable formal model $\mathfrak{X}$ of $X^{\textup{an}}$. Suppose that $\mathfrak{X}$ is algebraic, $\overline{L}$ has a model on $\mathfrak{X}$ and $\lambda \cdot\boldsymbol{dx}\leq \mu\leq\Lambda\cdot\boldsymbol{dx}$ on $\tau$ for some $\lambda,\Lambda >0$ where $\boldsymbol{dx}$ denotes the Lebesgue measure on $\tau$. Assume that $\varphi=\varphi\circ p_{\mathfrak{X}}$. Then $\varphi\in W^{2,1}_{loc}(\tau)$.  
		\end{Proposition}
		\begin{proof}
			By Corollary \ref{semipositive->convex} $\varphi$ is convex on every closed face of $\Delta$. Note that the metric on $\overline{L}$ is trivial on $p_{\mathfrak{X}}^{-1}(\tau)$. Hence we can apply Corollary \ref{mainresult} to get
			\[
				\mu=c_1(\overline{L} \otimes \overline{\mathcal{O}}^\varphi)^n =\Deg(S)\cdot n!\cdot\MA(\varphi)
			\]
			on $\tau$ where $S$ is the stratum of $\tilde{\mathfrak{X}}$ corresponding to $\tau$. Now the claim follows from the corresponding fact in the real case \cite[Theorem 1.2]{M}.
		\end{proof}
		\begin{Remark}
			The condition $\varphi=\varphi\circ p_{\mathfrak{X}}$ is not automatic as shown by a counterexample of Burgos and Sombra, see \cite[Appendix A]{GJKM}.
		\end{Remark}
		\begin{Proposition}
			\label{regularitycurves} Let $X$ be a smooth projective curve over $K$ and $\overline{L}$ a line bundle with a fixed formal metric. Let $\mu$ be a positive Borel measure on $X^{\textup{an}}$ and $\varphi$ a continuous function on $X^{\textup{an}}$ such that the metric on $\overline{L}\otimes\overline{\mathcal{O}}^\varphi$ is semipositive and solving the equation
			\[
				c_1(\overline{L} \otimes \overline{\mathcal{O}}^\varphi)=\mu.
			\]
			If $\tau$ is an open face of the skeleton $\Delta$ of a strictly semistable algebraic model $\mathscr{X}$ of $X^{\textup{an}}$ on which $\overline{L}$ has an algebraic model, $\mu$ is supported on $\Delta$ and $\mu=f\cdot\boldsymbol{dx}$ on $\tau$ for some positive function $f\in C^k(\tau)$ where $\boldsymbol{dx}$ denotes the Lebesgue measure on $\tau$ then $\varphi\in C^{k+2}(\tau)$.
		\end{Proposition}
		\begin{proof}
			By \cite[Proposition 1.2]{GJKM} we have $\varphi=\varphi\circ p_{\mathfrak{X}}$. As in the previous result $\varphi$ is convex on $\tau$ and 
			\[
				\mu=c_1(\overline{L} \otimes \overline{\mathcal{O}}^\varphi)=\Deg(S)\cdot\MA(\varphi)
			\]
			on $\tau$. But a solution to the archimedean Monge-Ampère problem is given by a second antiderivative of $f$ and the solution is unique up to addition of a linear function. Hence $\varphi\in C^{k+2}(\tau)$ and $\Deg(S)\cdot \varphi''=f$.
		\end{proof}
		\begin{appendix}
			\section{Reduction of germs}
			In this appendix we will explain the reduction of germs due to Michael Temkin (see \cite{T1} and \cite{T2}). At the end we will use this theory to prove a generalization of \cite[Lemme 6.5.1]{ChD} proposed by Antoine Ducros which drops a separatedness assumption.
			\begin{Def}
				 \begin{enumerate}
				 		\item The category of \textit{punctual strictly $K$-analytic spaces} is the following: The objects are pairs $(X,x)$ where $X$ is a strictly $K$-analytic space and $x\in X$ is a point. A morphism $\varphi:(X,x)\rightarrow(Y,y)$ is a morphism $\varphi:X\rightarrow Y$ of strictly $K$-analytic spaces such that $\varphi(x)=y$.
				 		\item The category $(K\textrm{-Germs})$ of germs of a strictly $K$-analytic space at a point is defined to be the localization of the category of punctual strictly $K$-analytic spaces by the system of morphisms $\varphi:(X,x)\rightarrow(Y,y)$ which identify $X$ with an open neighbourhood of $y$ in $Y$. The germ induced by the punctual strictly $K$-analytic space $(X,x)$ is denoted by $X_x$.
				 		\item A germ $X_x$ is said to be \textit{good} if $x$ has a strictly $K$-affinoid neighbourhood in $X$. A morphism of germs $\varphi:X_x\rightarrow Y_y$ is said to be \textit{separated} resp. \textit{closed} if it is induced by a separated resp. boundaryless morphism $X'\rightarrow Y$ for an open neighbourhood $X'$ of $x$ in $X$ (recall that a morphism $\varphi:X\rightarrow Y$ of $K$-analytic spaces is called boundaryless if $X=\Int(X/Y)$, where the relative interior $\Int(X/Y)$ is defined to be the set of all $x\in X$ such that for any affinoid domain $V\subseteq Y$ with $\varphi(x)\in V$ there is an affinoid neighbourhood $U\subseteq\varphi^{-1}(V)$ of $x$ in $\varphi^{-1}(V)$ such that $x\in\Int(U/V)$).
				 \end{enumerate}
			\end{Def}
			\begin{Def}
				Let $k$ be a field and let $L$ be a field extension of $k$.
				\begin{enumerate}
					\item The Zariski-Riemann space $\boldsymbol{P}_{L/k}$ is the set of valuation rings in $L$ which contain $k$ and whose quotient field is $L$ endowed with the coarsest topology such that all sets of the form $\boldsymbol{P}_{L/k}\{f\}:=\left\{R\in\boldsymbol{P}_{L/k}\;\Big|\;f\in R\right\}$ with $f\in L$ are open.
					\item The category $(\textrm{bir}_k)$ is the following: The objects are triples $(X,L,\phi)$ where $X$ is a connected quasi-compact and quasi-separated topological space, $L$ is a field extension of $k$ and $\phi:X\rightarrow\boldsymbol{P}_{L/k}$ is a local homeomorphism. A morphism $(X,L,\phi)\rightarrow(Y,M,\psi)$ is a pair $(h,i)$ where $h:X\rightarrow Y$ is a continuous map and $i:M\rightarrow L$ is a morphism of field extensions of $k$ such that $\psi\circ h=i^\#\circ\phi$ where $i^\#:\boldsymbol{P}_{L/k}\rightarrow\boldsymbol{P}_{M/k}$ is the morphism induced by $i$.
					\item A morphism $(h,i):(X,L,\phi)\rightarrow(Y,M,\psi)$ is called \textit{proper} if the map $X\rightarrow Y\times_{\boldsymbol{P}_{M/k}}\boldsymbol{P}_{L/k}$ is bijective. 
				\end{enumerate}
			\end{Def}
			In \cite[§2]{T1} Temkin introduced a reduction functor $\red$ from $(K\textrm{-Germs})$ to $(\textrm{bir}_{\tilde{K}})$ sending a germ $X_x$ to its reduction $\tilde{X_x}$. It can be described as follows (see \cite[§4]{T2}): If $X_x$ is a good germ, we can assume $X=\mathscr{M}(A)$ for a strictly $K$-affinoid algebra $A$. Then the character $\chi_x:A\rightarrow\mathscr{H}(x)$ induces a morphism $\tilde{\chi_x}:\tilde{A}\rightarrow\widetilde{\mathscr{H}(x)}$. Then $\tilde{X_x}=(\boldsymbol{P}_{\widetilde{\mathscr{H}(x)}/\tilde{K}}\{\tilde{\chi_x}(\tilde{A})\},\widetilde{\mathscr{H}(x)},\iota)$ where $\boldsymbol{P}_{\widetilde{\mathscr{H}(x)}/\tilde{K}}\{\tilde{\chi_x}(\tilde{A})\}$ is the set of all $R\in\boldsymbol{P}_{\widetilde{\mathscr{H}(x)}/\tilde{K}}$ for which $\tilde{\chi}(\tilde{A})\subseteq R$ and $\iota$ is the canonical embedding. If $X_x$ is separated one covers $X_x$ by finitely many good germs $V_x^i$. Then the germs $V_x^i\cap V_x^j$ are good and one obtains an open embedding $\widetilde{V_x^i\cap V_x^j}\rightarrow \tilde{V_x^i}$. In fact this gives a glueing data and $\tilde{X}_x$ is the space obtained by glueing the $\tilde{V_x^i}$ along these open embeddings. Lastly if $X_x$ is arbitrary, one covers $X_x$ by finitely many separated germs $V_x^i$ and again gets open embeddings $\widetilde{V_x^i\cap V_x^j}\rightarrow\tilde{V_x^i}$ along which the $\tilde{V_x^i}$ are glued to $\tilde{X_x}$.
			\begin{Proposition}
				\label{interiorproper} Let $\mathfrak{X}$ be an admissible formal scheme and $x\in X:=\mathfrak{X}^{\textup{an}}$. Let $V$ be the closure of $\{\red(x)\}$ in the special fibre $\tilde{\mathfrak{X}}$. Then $V$ is proper if and only if the morphism $\tilde{X_x}\rightarrow\boldsymbol{P}_{\widetilde{\mathscr{H}(x)}/\tilde{K}}$ is bijective.
			\end{Proposition}
			\begin{proof}
				Let $(Y_i)_{i\in I}$ be an open affine cover of $V$ and set $V^i:=\red^{-1}(Y_i)$. Then $V^i$ is strictly $K$-affinoid by \cite[Theorem 3.1]{Bo2} and hence $V^i_x$ is a good germ. Note that $(V^i_x)_{i\in I}$ is a cover of $X_x$. Hence $\tilde{X_x}$ is obtained by glueing the $\tilde{V^i_x}$ along the canonical maps $\widetilde{V^i_x\cap V^j_x}\rightarrow\tilde{V^i_x}$. Let $V^i=\mathscr{M}(A_i)$ for a strictly $K$-affinoid algebra $A_i$. Then $Y_i=\Spec{\tilde{A_i}}$ and the character $\chi_x:A_i\rightarrow\mathscr{H}(x)$ induces a morphism $\tilde{\chi_x}:\tilde{A_i}\rightarrow\widetilde{\mathscr{H}(x)}$. Let $\mathfrak{p}\subseteq\tilde{A_i}$ be the prime ideal corresponding to $\red(x)$ i.e. $\mathfrak{p}$ is the kernel of $\tilde{\chi_x}$. The induced morphism $\tilde{A_i}/\mathfrak{p}\rightarrow\widetilde{\mathscr{H}(x)}$ is injective and hence it extends to a morphism $\tilde{K}(V)\rightarrow\widetilde{\mathscr{H}(x)}$ where $\tilde{K}(V)=\Quot(\tilde{A_i}/\mathfrak{p})$ denotes the function field of $V$. This induces a morphism $\pi:\boldsymbol{P}_{\widetilde{\mathscr{H}(x)}/\tilde{K}}\rightarrow\boldsymbol{P}_{\tilde{K}(V)/\tilde{K}}$. \\[.5cm]
				\textit{First step:} We have that $\tilde{V^i_x}=\boldsymbol{P}_{\widetilde{\mathscr{H}(x)}/\tilde{K}}\{\tilde{\chi_x}(\tilde{A_i})\}$ is the preimage under $\pi$ of the set of valuation rings in $\tilde{K}(V)$ which admit a center on $Y_i\cap V$. \\[.5cm]
				Indeed if $R\in\boldsymbol{P}_{\widetilde{\mathscr{H}(x)}/\tilde{K}}$ is a valuation ring with $\tilde{\chi_x}(\tilde{A_i})\subseteq R$ then $\tilde{A_i}/\mathfrak{p}\subseteq R\cap \tilde{K}(V)$. Let $\mathfrak{m}_R$ be the maximal ideal of $R$ then $\mathfrak{p}':=\mathfrak{m}_R\cap\tilde{A_i}/\mathfrak{p}$ defines a point in $\Spec(\tilde{A_i}/\mathfrak{p})$ whose local ring is $(\tilde{A_i}/\mathfrak{p})_{\mathfrak{p}'}$ and we have $(\tilde{A_i}/\mathfrak{p})_{\mathfrak{p}'}\subseteq R$. Then $R\cap \tilde{K}(V)$ admits the center $\mathfrak{p}'$ on $Y_i\cap V$ as claimed. Conversely if $R\cap \tilde{K}(V)$ admits a center on $Y_i\cap V$ then there exists $\mathfrak{p}'\in\Spec(\tilde{A_i}/\mathfrak{p})$ such that $(\tilde{A_i}/\mathfrak{p})_{\mathfrak{p}'}\subseteq R\cap\tilde{K}(V)$ and hence obviously $\tilde{\chi_x}(\tilde{A_i})\subseteq R$. \\[.5cm]
		\textit{Second step:} The map $\tilde{X_x}\rightarrow\boldsymbol{P}_{\widetilde{\mathscr{H}(x)}/\tilde{K}}$ is surjective if and only if any valuation on $\tilde{K}(V)/\tilde{K}$ admits at least one center on $V$. \\[.5cm]
		Let $\tilde{X_x}\rightarrow\boldsymbol{P}_{\widetilde{\mathscr{H}(x)}/\tilde{K}}$ be surjective and $v$ a valuation on $\tilde{K}(V)/\tilde{K}$. Then $v$ extends to a valuation $\tilde{v}$ on $\widetilde{\mathscr{H}(x)}$. Let $R$ be the valuation ring of $\tilde{v}$. Then $R\in\boldsymbol{P}_{\widetilde{\mathscr{H}(x)}/\tilde{K}}$ and hence $R$ has a preimage $R'\in\tilde{X_x}$. Then there exists $i\in I$ such that $R'\in \tilde{V^i_x}$ hence the image of $R'$ in $\boldsymbol{P}_{\tilde{K}(V)/\tilde{K}}$ admits a center on $Y_i\cap V$ by the first step. But this image is $R\cap \tilde{K}(V)$ by construction which is the valuation ring of $v$. Hence $v$ admits a center on $V$. Conversely suppose that any valuation on $\tilde{K}(V)/\tilde{K}$ admits a center on $V$ and let $R\in\boldsymbol{P}_{\widetilde{\mathscr{H}(x)}/\tilde{K}}$ then the image of $R$ in $\boldsymbol{P}_{\tilde{K}(V)/\tilde{K}}$ induces a valuation on $\tilde{K}(V)/\tilde{K}$ which admits a center $z$ on $V$. Let $i\in I$ such that $z\in Y_i$ then $R\in\tilde{V^i_x}$ and the induced element in $\tilde{X_x}$ is a preimage of $R$. \\[.5cm]
				\textit{Third step:} The map $\tilde{X_x}\rightarrow\boldsymbol{P}_{\widetilde{\mathscr{H}(x)}/\tilde{K}}$ is injective if and only if every valuation on $\tilde{K}(V)/\tilde{K}$ admits at most one center on $V$. \\[.5cm]
				To see this we describe $\widetilde{V^i_x\cap V^j_x}$. In order to do so we cover $Y_i\cap Y_j$ by open affine subsets $Y_{i,j}^k$. Their preimages under $\red$ yield a cover of $V^i_x\cap V^j_x$ by good germs. As above their reductions can be described as the preimage of the set of valuation rings in $\tilde{K}(V)$ which admit a center on $Y_{i,j}^k\cap V$. The reduction of $V^i_x\cap V^j_x$ is then obtained by glueing these spaces. Now suppose that any valuation on $\tilde{K}(V)/\tilde{K}$ admits at most one center and let $R_1,R_2\in\tilde{X_x}$ which map to the same valuation ring $R\in\boldsymbol{P}_{\widetilde{\mathscr{H}(x)}/\tilde{K}}$. There exists $i,j$ such that $R_1\in\tilde{V^i_x},R_2\in\tilde{V^j_x}$. As we have seen in the first step, $R_1\cap\tilde{K}(V)$ and $R_2\cap\tilde{K}(V)$ admit centers $y_1\in\Spec(\tilde{A_i})\cap V$ respectively $y_2\in\Spec(\tilde{A_j})\cap V$. Then both are a center of $R\cap \tilde{K}(V)$. Hence $y_1=y_2\in Y_i\cap Y_j$ by our assumption. Therefore by the first step $R_1=R_2=R$ in $\widetilde{V^i_x\cap V^j_x}$. Hence in the glueing process, $R_1$ and $R_2$ are identified with each other. Conversely suppose that there is a valuation on $\tilde{K}(V)/\tilde{K}$ which admits two centers $y_1,y_2\in V$. Let $y_1\in Y_i$ and $y_2\in Y_j$. Choose an extension of the valuation to $\widetilde{\mathscr{H}(x)}$ and let $R$ denote its valuation ring. Then $R$ induces an element $R_1\in \tilde{V^i_x}$ as well as an element $R_2\in\tilde{V^j_x}$. Then $R_1$ and $R_2$ map to the same element $R$ in $\boldsymbol{P}_{\widetilde{\mathscr{H}(x)}/\tilde{K}}$ but they are not identified in the glueing process as $Y_i\cap V$ and $Y_j\cap V$ are separated and hence $R_1$ and $R_2$ admit at most one center in $\Spec(\tilde{A_i})\cap V$ respectively $\Spec(\tilde{A_j})\cap V$ which means in particular that they do not admit a center in $Y_i\cap Y_j\cap V$. Hence $\tilde{X_x}\rightarrow\boldsymbol{P}_{\widetilde{\mathscr{H}(x)}/\tilde{K}}$ is not injective. This proves the third step. \\[.5cm]
				Recall that $V$ is proper if and only if every valuation on $\tilde{K}(V)/\tilde{K}$ admits a unique center on $V$ (\cite[Ch. II, Ex. 4.5]{Ha}). Hence the claim follows from the second and third step.
			\end{proof}
			\begin{Corollary}
				\label{CDLemma} In the situation of Proposition \ref{interiorproper}, $x$ is an interior point of $X$ if and only if $V$ is proper.
			\end{Corollary}
			\begin{proof}
				By Proposition \ref{interiorproper}, $V$ is proper if and only if the map $\tilde{X_x}\rightarrow\boldsymbol{P}_{\widetilde{\mathscr{H}(x)}/\tilde{K}}$ is bijective which by \cite[Theorem 5.2]{T2} is equivalent to the map $X_x\rightarrow\mathscr{M}(K)$ being closed. But this is equivalent to $x$ being an interior point of $X$.
			\end{proof}
		\section{Convexity of psh-functions}
			In order to be able to use the results from section \ref{sectioncomparison} we need that semipositive metrics lead to convex functions on the faces of some skeleton. The proof of this is based on the proof of \cite[Proposition 7.5]{BFJ2}, where this is done for SNC models and discretely valued $K$ with residue characteristic zero, and unpublished work of Walter Gubler and Florent Martin.
		\begin{Lemma}
			\label{log|f|} Let $\mathfrak{X}$ be a strongly nondegenerate strictly polystable formal scheme with associated skeleton $\Delta$ and $f\in\mathcal{O}(\mathfrak{X}^{\textup{an}})$ such that $\left\{x\in\mathfrak{X}^{\textup{an}}\;\Big|\;f(x)=0\right\}$ is nowhere dense. Then for any $x\in\Delta$ we have $|f(x)|\neq 0$ and the function $\varphi:\mathfrak{X}^{\textup{an}}\rightarrow\mathbb{R}\cup\{-\infty\}$ given by $\varphi(x):=\log|f(x)|$ is piecewise affine linear and convex on each face of $\Delta$ and satisfies $\varphi\leq\varphi\circ p_{\mathfrak{X}}$.
		\end{Lemma}
		\begin{proof}
			By \cite[Theorem 5.1.1]{Be3} we know that $|f(x)|\neq 0$ and that $\varphi$ is piecewise affine linear on $\Delta$. By \cite[Theorem 5.2]{Be2} we have $\varphi\leq\varphi\circ p_{\mathfrak{X}}$. Assume there is a face $\tau$ of $\Delta$ on which $\varphi$ is not convex, i.e. there are $x,y\in \tau$ and $t\in(0,1)$ such that
			\[
				\delta:=\varphi(tx+(1-t)y)-t\varphi(x)-(1-t)\varphi(y)>0.
			\]
			By base change we can assume that $K$ is algebraically closed and then by density of the value group $\Gamma$ and continuity of $\varphi$ that the coordinates of $x$ and $y$ are in $\Gamma$. Choose a $\Gamma$-rational polytopal subdivision of $\Delta$ which only has $x$ and $y$ as additional vertices. By Construction \ref{Formalmodel} we get an admissible formal model $\mathfrak{X}''$ of $\mathfrak{X}^{\textup{an}}$ dominating $\mathfrak{X}$. Choose an affine open $U\subseteq\tilde{\mathfrak{X}}''$ which contains $\red(tx+(1-t)y)$. By the stratum face correspondence (Proposition \ref{Stratumface} and Corollary \ref{Corstratumface}) the vertices $x$ and $y$ correspond to irreducible components of $\tilde{\mathfrak{X}}''$. By taking out all other irreducible components we may assume that $U$ intersects only those corresponding to $x$ and $y$. Then $V:=\red^{-1}(U)$ is a strictly $K$-affinoid domain by \cite[Theorem 3.1]{Bo2}. By \cite[Proposition 1.4]{Be2} its canonical reduction has two irreducible components, namely those corresponding to $x$ and $y$. Hence the Shilov boundary of $V$ is the set $\{x,y\}$ by \cite[Proposition 2.4.4]{Be1} and we get $|f(tx+(1-t)y)|\leq\max\left\{|f(x)|,|f(y)|\right\}$. Since $x\neq y$, by restricting to a building block $\mathfrak{U}$, we can find a coordinate function $g\in\mathcal{O}(\mathfrak{U}^{\textup{an}})^\times$ such that $|g(x)|\neq |g(y)|$. Then we can find $N\in\mathbb{N}_{>0}$ and $m\in\mathbb{Z}$ such that
			\[
				\Big|\log|f^Ng^m(x)|-\log|f^Ng^m(y)|\Big|=\Big|N(\varphi(x)-\varphi(y))+m(\log|g(x)|-\log|g(y)|)\Big|< N\delta.
			\]
			Since $\log|g^m|$ is affine linear on $\tau$ we get
			\[
				\log|f^Ng^m(tx+(1-t)y)|-t\log|f^Ng^m(x)|-(1-t)\log|f^Ng^m(y)|=N\delta.
			\]
			Hence by replacing $f$ with $f^Ng^m$ and $\delta$ by $N\delta$ we can assume
			\[
				\delta:=\varphi(tx+(1-t)y)-t\varphi(x)-(1-t)\varphi(y)>0.
			\]
			and
			\[
				|\varphi(x)-\varphi(y)|<\delta.
			\]
			Then
			\[
				\varphi(tx+(1-t)y)=\delta+t\varphi(x)+(1-t)\varphi(y)> t\varphi(x)+(1-t)\varphi(y)+|\varphi(x)-\varphi(y)|.
			\]
			Now on the one hand we have
			\[
				t\varphi(x)+(1-t)\varphi(y)+|\varphi(x)-\varphi(y)|\geq t\varphi(x)+(1-t)\varphi(y)+t(\varphi(y)-\varphi(x))=\varphi(y)
			\]
			while on the other hand
			\[
				t\varphi(x)+(1-t)\varphi(y)+|\varphi(x)-\varphi(y)|\geq\varphi(x)+t(\varphi(x)-\varphi(y)).
			\]
			Together we get
			\[
				\varphi(tx+(1-t)y)>\max\left\{\varphi(x),\varphi(y)\right\}.
			\]
			But this violates our previous observation that $|f(tx+(1-t)y)|\leq\max\left\{|f(x)|,|f(y)|\right\}$. This finishes the proof.  
		\end{proof}
		\begin{Def}
			Let $\mathscr{X}$ be an algebraic scheme over $K^\circ$, $\mathfrak{a}$ a vertical coherent fractional ideal sheaf on $\mathscr{X}$ (i.e. $\mathfrak{a}$ is a coherent subsheaf of the sheaf of total quotient rings $\mathcal{K}_{\mathscr{X}}$ such that after multiplying with some element of $K^\circ\setminus\{0\}$ it becomes a vertical ideal sheaf) and $\red:\mathscr{X}^{\textup{an}}\rightarrow\tilde{\mathfrak{X}}$ the reduction map. We define the function $\log|\mathfrak{a}|:\mathscr{X}^{\textup{an}}\rightarrow\mathbb{R}$ by $\log|\mathfrak{a}|(x):=\sup\left\{\log|f(x)|\;\Big|\;f\in\mathfrak{a}_{\red(x)}\right\}$. The supremum is actually a maximum as for a set of generators $f_1,...,f_r$ of $\mathfrak{a}_{\red(x)}$ we have $\sup\left\{\log|f(x)|\;\Big|\;f\in\mathfrak{a}_{\red(x)}\right\}=\max\left\{\log|f_i(x)|\;\Big|\;1\leq i\leq r\right\}$.
		\end{Def}
		\begin{Lemma}
			\label{approx} Let $X$ be a proper scheme over $K$ and $\overline{L}$ a line bundle on $X$ with an algebraic metric $\|\cdot\|_{\overline{L}}$. Let $\|\cdot\|$ be a piecewise $\mathbb{Q}$-linear metric on $\mathcal{O}_{X^{\textup{an}}}$ such that $\|\cdot\|_{\overline{L}}\otimes\|\cdot\|$ is a semipositive piecewise $\mathbb{Q}$-linear metric. Let $\mathscr{X}$ be an algebraic model of $X$ such that $\overline{L}$ has a model $\mathscr{L}$ on $\mathscr{X}$ and set $\varphi:=-\log\|1\|$. Then there is a sequence $(\mathfrak{a}_n)_{n\in\mathbb{N}}$ of vertical coherent fractional ideals on $\mathscr{X}$ and a sequence $(d_n)_{n\in\mathbb{N}}$ of positive integers such that $\frac{1}{d_n}\log|\mathfrak{a}_n|$ converges uniformly to $\varphi$.
		\end{Lemma}
		\begin{proof}
			We may assume that $\|\cdot\|$ is a piecewise linear metric. Let $\mathscr{X}'$ be an algebraic model of $X$ on which $(\mathcal{O}_{X^{\textup{an}}},\|\cdot\|)$ has an algebraic model $\mathscr{M}$. The section $1$ of $\mathcal{O}_X$ extends to a meromorphic section $s$ of $\mathscr{M}$ and then $\mathscr{M}=\mathcal{O}(D)$ for the vertical Cartier divisor $D=\Div(s)$ on $\mathscr{X}'$. By \cite[Theorem 13.98]{GW} we may assume that $\mathscr{X}'$ is a vertical blowup of $\mathscr{X}$. Denote by $\pi$ the canonical map $\mathscr{X}'\rightarrow\mathscr{X}$. We show first that $D$ is $\pi$-nef, i.e. $\deg(D\cdot C)\geq 0$ for any closed curve $C\subseteq\tilde{\mathscr{X}}'$ which is contracted by $\pi$. \\
			So let $x\in\tilde{\mathscr{X}}$ be a closed point and $C\subseteq\pi^{-1}(x)$ a curve. Then by the semipositivity assumption $\deg((\mathcal{O}(D)+\pi^{\ast}\mathscr{L})\cdot C)\geq 0$. But since $\pi_{\ast}(\pi^{\ast}\mathfrak{L}\cdot C)=\mathscr{L}\cdot\pi_{\ast}(C)=0$ we have $\deg(\pi^{\ast}\mathscr{L}\cdot C)=0$ and hence $\deg(D\cdot C)\geq 0$. \\
			Now let $A$ be a $\pi$-ample vertical Cartier divisor on $\mathscr{X}'$, e.g. $A=-E$ for the exceptional divisor $E$ of the blowup (this is $\pi$-ample by \cite[Proposition 13.96]{GW}). Then $D+A$ is $\pi$-ample by the relative version of Kleiman's criterion (\cite[Remark 7.41]{D}). Furthermore, since $\mathcal{O}_{\mathscr{X}'}(D)$ and $\mathcal{O}_{\mathscr{X}'}(A)$ are coherent vertical fractional ideal sheaves, also $\mathfrak{a}:=\pi_{\ast}\mathcal{O}_{\mathscr{X}'}(m(D+A))$ is a coherent vertical fractional ideal sheaf on $\mathscr{X}$ for any $m\in\mathbb{N}_{>0}$ by \cite[Theorem 5.3]{U}. \\
			By the characterization of $\pi$-ampleness in \cite[Proposition 4.6.8]{EGAII} there exists some $m\in\mathbb{N}_{>0}$ such that $\pi^{\ast}\mathfrak{a}\rightarrow\mathcal{O}_{\mathscr{X}'}(m(D+A))$ is surjective. This implies 
			\[
				\log|\mathfrak{a}|=\log|\pi^{\ast}\mathfrak{a}|=\log|\mathcal{O}_{\mathscr{X}'}(m(D+A))|=m\cdot(\varphi-\log\|1\|_{\mathcal{O}_{\mathscr{X}'}(A)})
			\]
			and hence $\frac{1}{m}\log|\mathfrak{a}|=\varphi-\log\|1\|_{\mathcal{O}_{\mathscr{X}'}(A)}$. Since we can replace $A$ by $\epsilon A$ for arbitrary small $\epsilon\in\mathbb{Q}_{>0}$ this concludes the proof.	    
		\end{proof}
		\begin{Corollary}
			\label{semipositive->convex} In the situation of Lemma \ref{approx} suppose that the formal completion $\mathfrak{X}$ of $\mathscr{X}$ is strongly nondegenerate strictly polystable and denote by $\Delta$ the associated skeleton. Then $\varphi$ is convex on every face of $\Delta$ and satisfies $\varphi\leq\varphi\circ p_{\mathfrak{X}}$.
		\end{Corollary}
		\begin{proof}
			By Lemma \ref{approx} we may approximate $\varphi$ by functions of the form $\frac{1}{d_m}\log|\mathfrak{a}_m|$ for some vertical coherent fractional ideals $\mathfrak{a}_m$ on $\mathscr{X}$. On the generic fibre of a building block $\mathfrak{U}$, the function $\log|\mathfrak{a}_m|$ is given as the maximum of the functions $\log|f|$ where $f$ runs through a finite set of generators of $\mathfrak{a}_m\Big|_{\mathfrak{U}}$. Since the properties we are looking for are stable under taking the maximum, these functions have them by Lemma \ref{log|f|}. But they are also stable under uniform limits so we are done.
		\end{proof}
		\end{appendix} 
	%\emergencystretch=1em
	\begingroup
	%\RaggedRight
	%\tolerance=1062
	%\hbadness=1062
	\enlargethispage{2\baselineskip}
	\bibliographystyle{alpha}
	\bibliography{mybib}
	%\printbibliography%[heading=bibintoc]
	\endgroup
		
\end{document}